\documentclass[11pt]{amsart}
\usepackage[left=1.5in, right=1.5in]{geometry}
\usepackage[latin1]{inputenc}
\usepackage{amsfonts}
\usepackage{color}
\usepackage[toc,page,title,titletoc,header]{appendix}
\usepackage{graphicx,psfrag,epsfig,multirow,hyperref}
\usepackage{amssymb,amsmath,amscd,amsthm,verbatim}
\newtheorem{Thm}{Theorem}[section]
\newtheorem{Lm}[Thm]{Lemma}
\newtheorem{Prop}[Thm]{Proposition}
\newtheorem{Cor}[Thm]{Corollary}
\newtheorem{Def}[Thm]{Definition}
\newtheorem{Rk}{Remark}[section]

\usepackage{soul}

\newcommand{\gm}{\gamma}

\newcommand{\al}{\alpha}

\newcommand{\bn}{\mathbf{n}}

\newcommand{\cW}{\mathcal{W}}

\newcommand{\cF}{\mathcal{F}}
\newcommand{\cE}{\mathcal{E}}

\newcommand{\cK}{\mathcal{K}}

\newcommand{\cN}{\mathcal{N}}
\newcommand{\cM}{\mathcal{M}}
\newcommand{\cS}{\mathcal{S}}

\newcommand{\cQ}{\mathcal{Q}}
\newcommand{\cP}{\mathcal{P}}
\newcommand{\cX}{\mathcal{X}}

\newcommand{\R}{\mathbb{R}}

\newcommand{\dt}{\delta}

\newcommand{\eps}{\varepsilon}

\newcommand{\dv}{\mathrm{div}}
\newcommand{\Hess}{\mathrm{Hess}}
\newcommand{\Ric}{\mathrm{Ric}}

\title[Dynamics of MCF]{Generic dynamics of Mean Curvature Flows with closed singularities}
\author{Ao Sun}
\email{aos223@lehigh.edu}
\address{Department of Mathematics,
	University of Chicago,
	5734 S. University Avenue,
	Chicago, IL 60637, USA
	\\
	Current address: 
	Lehigh University, Department of Mathematics, Chandler-Ullmann Hall, Bethlehem, PA 18015
}
\author{Jinxin Xue}
\email{jxue@tsinghua.edu.cn}
\address{New Cornerstone Sciences Laboratory, Department of Mathematics, Room A115, Science Building, Tsinghua University, Beijing, China, 100084}

\begin{document}
	\maketitle
	
	\begin{abstract}
		This is a contribution to the program of the dynamical approach to mean curvature flow initiated by Colding and Minicozzi. In this paper, we prove two main theorems. The first one is local in nature, and the second one is global. In this first theorem, we apply the invariant manifold theory from hyperbolic dynamics to study the dynamics close to a closed shrinker that is not a sphere. In the second theorem, we show that if a hypersurface under the rescaled mean curvature flow converges to a closed shrinker that is not a sphere, then a Baire generic perturbation on initial data would make the flow leave a small neighborhood of the shrinker and never come back. The key is to prove that a positive perturbation would drift to the first eigenfunction direction under the linearized equation. This result can be viewed as a global unstable manifold theorem in the most unstable direction.
	\end{abstract}
	\section{Introduction}
	
	A family of hypersurfaces $\{\mathbf{M}_t\}\subset\R^{n+1}$ flows by mean curvature if they satisfy the equation
	\begin{equation*}
		\partial_t x=-H\bn.
	\end{equation*}
	Here $x$ is the position vector of $\mathbf{M}_t$, $H$ is the mean curvature of $\mathbf{M}_t$, and $\bn$ is the outer unit normal vector field on $\mathbf{M}_t$.
	
	Mean curvature flow (MCF) of closed hypersurfaces must generate finite time singularities, and the singularities are modeled by the tangent flow, see \cite{Huisken90}, \cite{White94_PartialRegMeanConvex}, \cite{Ilmanen95_Sing2D}. The tangent flow at the singular space-time point can be characterized by the rescaled mean curvature flow (RMCF), which is a family of hypersurfaces $M_t$ satisfying the equation
	\begin{equation*}
		\partial_t x=-\left(H-\frac{\langle x,\bn\rangle}{2}\right)\bn.
	\end{equation*}
	In particular, shrinkers, which are hypersurfaces satisfying the equation
	$H=\frac{\langle x,\bn\rangle}{2}$
	are static under RMCF. MCF and RMCF are equivalent to each other up to a space-time rescaling. In this paper, we only consider MCF and RMCF of closed embedded hypersurfaces.
	
	In this paper, we investigate the behaviour of MCF near a singularity modeled by a closed shrinker from a dynamical point of view. Huisken's monotonicity formula shows that the RMCF is the negative gradient flow of the following Gaussian area functional
	\[\cF(M)=\int_M e^{-\frac{|x|^2}{4}}d\mu \]
	and shrinkers are fixed points of the RMCF. For negative gradient flows, of central importance are their fixed points. For instance, in Morse theory, fixed points of a generic function on a closed manifold can be used to generate the homology group of the manifold. Here, when studying the RMCF, we are working on negative gradient flow on the infinite-dimensional space of hypersurfaces, which is much more complicated than the finite-dimensional case. 
	
	The first step for a dynamical study of the RMCF is the analysis of the stability of shrinkers. Such a study was initiated in \cite{ColdingMinicozzi12_generic}, where the linear (in)stability and variational (in)stability of shrinkers are studied. For instance, it is known that the linearized RMCF equation at a shrinker $\Sigma$ has the form $\partial_t u=L_\Sigma u$ where \[
	L_\Sigma=\Delta_\Sigma-\frac{1}{2}\langle x,\nabla_\Sigma\ \cdot\rangle +(|A|^2+1/2)
	\]
	is self-adjoint with respect to the Gaussian $L^2$-inner product $\langle u,v\rangle=\int_\Sigma u(x)v(x)e^{-\frac{|x|^2}{4}}d\mu$, where $A$ is the second fundamental form.  Moreover, \cite{ColdingMinicozzi12_generic} showed that in the compact case, unless $\Sigma$ is a sphere, $L_\Sigma$ has the leading eigenvalue greater than $1$, whose corresponding eigenfunction is an exponentially growing mode for the linearized equation. It follows that the sphere is the only linearly stable compact shrinker. 
	Concerning the dynamical (in)stability, i.e. the (in)stability of shrinkers under the evolution of the nonlinear RMCF instead of the linearized RMCF, the study was continued in \cite{ColdingMinicozzi19_Dynamics}, where the authors studied nonspherical compact shrinkers and proved the existence of certain stable set of positive codimension, outside of which the flow leaves a neighborhood of the shrinker, and in \cite{ColdingMinicozzi21_wandering}, Colding-Minicozzi studied the wandering property of these shrinkers. %

	\subsection{Avoiding unstable singularities}
	By analogy with the Morse theory, we see that provided an initial condition does not lie on the stable manifolds of the set of critical points of a Morse function with positive Morse indices, its evolution under the negative gradient flow will end up with the critical point with $0$ Morse index. Note that the set of stable manifolds of critical points with positive Morse indices has positive codimension, thus a generic initial condition has $\omega$-limit set the critical point with $0$ Morse index. 
	Thus, we expect that a generic initial condition under the MCF avoids all nonspherical compact shrinkers. This turns out to be the case as stated in the next theorem. To state the next theorem, we recall the notion of entropy introduced by Colding-Minicozzi in \cite{ColdingMinicozzi12_generic}: the supremum of the Gaussian area of $M$ under all possible dilations and translations. i.e.
	\begin{equation*}
		\lambda(M)=\sup_{x\in\R^{n+1},t\in(0,\infty)}\cF(t^{-1}(M-x)).
	\end{equation*}
	By Huisken's monotonicity formula  (\cite{Huisken90}), the entropy of a hypersurface is non-increasing under mean curvature flow. Entropy is also lower semi-continuous in the space of hypersurfaces. Therefore, if $\Sigma$ is a shrinker and $\lambda(\Sigma)>\lambda(M_{t_0})$, then $\Sigma$ can never be the tangent flow of MCF $M_t$ starting from $M_{t_0}$.

	\begin{Thm}\label{ThmGlobalCpt}
		Let $\{M_t\}_{t\in[0,\infty)}$ be an RMCF with $M_t\to \Sigma$ smoothly as $t\to\infty$ where $\Sigma$ is a compact shrinker that is not a sphere. Then there exists an open dense subset $\mathcal S$ of $\{v_0\in C^{2,\al}(M_0)\ |\ \|v_0\|=1\}$, such that for any $v_0\in\mathcal S$, there exist $\eps_0>0$ such that for all $0<\eps<\eps_0$, there exists $T>0$ such that for the perturbed flow $\{\widetilde{M}_t\}$ starting from $\{x+\eps v_0(x)\mathbf n(x)\ |\ x\in M_0\}$, we have $\lambda(\widetilde M_{T})<\lambda(\Sigma).$ 
	\end{Thm}
	
	From the dynamical viewpoint, the last theorem is highly nontrivial. For a system of finite dimensional, with the local invariant manifold close to a fixed point, we can construct the global stable manifold by flowing the local stable manifold under the backward flow. By avoiding initial conditions from the stable manifold in general, we can avoid the fixed point if the linearized operator has no zero eigenvalue.

	However, the main difficulty in the case of RMCF (and other nonlinear heat equations) is that we are not allowed to run the flow backwardly, since a geometric heat flow is not reversible. Thus, it is not clear whether the global stable manifold exists or not. Our generic perturbation theorem has a global nature, since we have to control the dynamics from a neighborhood of the initial conditions until a neighborhood of the shrinker, where the two neighborhoods may be quite far. 
	
	From the perspective of mean curvature flows, it is known that even in $\R^3$ there are infinitely many shrinkers (see \cite{Nguyen14_Shrinker}, \cite{KapouleasKleeneMoller18_DesingShrinker}, \cite{SunWangZhou20_MinmaxShrinker}), which seem hopeless to classify (see \cite{Wang14_UniquenessConicalEnds}, \cite{wang2016asymptotic}). Thus, it is natural to introduce the idea of generic MCFs in hopes of avoiding those unstable shrinkers. This is the motivation for Huisken to propose his conjecture that generic MCFs in $\R^3$ have only spherical and cylindrical singularities (cf. \cite{AngenentIlmanenChopp95_Eg} in $\R^3$). 
	
	In the following, we use the word ``generic" in the sense of Theorem \ref{ThmGlobalCpt}. The equivalence of RMCF and MCF gives the following corollary.
	
	\begin{Cor}\label{Cor:mainCor}
		Suppose $\{\mathbf M_t\}_{t\in[0,T)}$ is an MCF and the first time singularity is characterized by a multiplicity $1$ closed shrinker $\Sigma$ which is not a sphere. Then we can perturb $\mathbf M_0$ generically to a nearby hypersurface $\widetilde{\mathbf M}_0$ such that the MCF $\{\widetilde{\mathbf M}_t\}$ starting from $\widetilde{\mathbf M}_0$ will never encounter a singularity characterized by $\Sigma$.
	\end{Cor}
	
	This result is an attempt to solve Conjecture 8.2 in \cite{CMP15_MCF}. In \cite{CMP15_MCF}, Colding-Minicozzi-Pedersen proposed that after a generic perturbation on the initial data of an MCF, the perturbed MCF will only encounter generic singularities. A consequence of our Corollary \ref{Cor:mainCor} is that, in $\R^3$, after a perturbation on the initial data, the perturbed MCF will encounter the first singularity, either non-compact, or has higher multiplicity, or a sphere.	\begin{Cor}\label{CorR3}
		Consider an MCF in $\R^3$. After a generic small perturbation on initial data, the first singularity of an MCF is modeled by either a sphere, a non-compact shrinker, or a higher multiplicity shrinker.
	\end{Cor}

	Compared with the previous work on generic perturbations of MCF like \cite{ColdingMinicozzi12_generic}, \cite{baldauf-sun_2018sharp}, \cite{sun-generic-multi-1}, our perturbation is applied to the initial data, while in the previous work the perturbations are applied to a moment very close to the singular time. Recently, Chodosh-Choi-Mantoulidis-Schulze also studied the perturbation of the initial data of an MCF in \cite{CCMS20_GenericMCF} and \cite{CCMS21_GenericMCF_LowEntropy}. In particular, they proved that after an initial positive perturbation, the perturbed MCF will avoid any singularities modeled by non-generic closed and conical shrinkers. Compared with their work, our initial perturbations are not necessarily positive (in their terminology, positive is ``one-sided"). Our work focuses more specifically on the dynamical behaviour of the MCF near the singularity, and how they bypass those non-generic singularities from a dynamical view. In \cite{SunXue2021_initial_conical}, we shall prove how to avoid noncompact shrinkers with conical ends by generic initial perturbations. The idea of this paper is further developed to study the generic dynamics of MCF with spherical and cylindrical singularities, see \cite{SunXue2022_generic_cylindrical, SunXue2022_generic_regularity}.

	\begin{figure}[h!]
		\centering
		\includegraphics[width=0.6\textwidth]{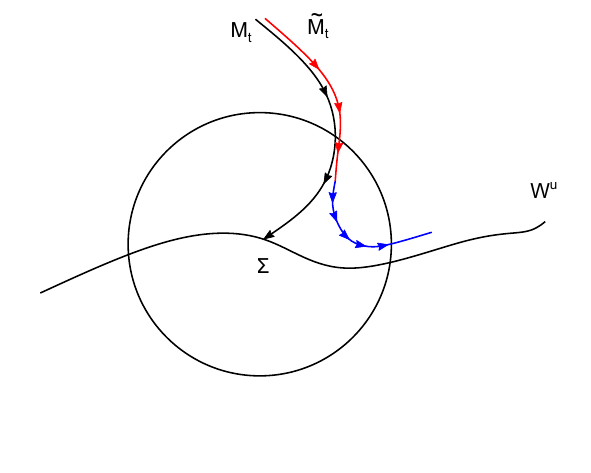}
		\caption{This figure explains the behaviour of the perturbed RMCF $\widetilde{M}_t$. In the figure, $M_t$ is an RMCF converging to the shrinker $\Sigma$. The circle is the $C^{2,\al}$ $\dt$-neighbourhood of $\Sigma$ where the hypersurfaces are graphical on large region of $\Sigma$. The red part of the orbit shows that $\widetilde{M}_t$ starts to move to the unstable manifold direction; the blue part of the orbit shows that $\widetilde{M}_t$ will move further to the unstable manifold direction, using the invariant cone argument discussed in Section \ref{S:3}.}
		\label{Figure}
	\end{figure}
	
	\subsection{Ideas of the proofs}
	
	As we have explained before, the invariant manifold theory is only local, and does not give us the existence of a global stable manifold, so the proof of Theorem \ref{ThmGlobalCpt} needs essential new ingredients. Suppose we know that an RMCF $(M_t)$ is converging to a closed shrinker, we have to analyze the dynamics of nearby orbits, which leads naturally to the variational equation along the orbit $(M_t)$
	$$\partial_t v=L_{M_t} v.$$
	This can be considered as the Jacobi field equation for the RMCF, which measures the difference of the RMCFs $(M_t)$ and $(\widetilde M_t)$ where the latter starts from $\widetilde M_0$ that is a normal graph of the function $\eps v(0)$ over $M_0$. This dynamical interpretation of the heat-type equation is crucial for our proof of  Theorem \ref{ThmGlobalCpt}. 
	
	The proof of  Theorem \ref{ThmGlobalCpt} then consists of the following two main steps. First, we develop a Li-Yau estimate of the variational equation to obtain a Harnack estimate for the positive solutions. When time $t=T$ is so large that $M_t$ is very close to $\Sigma$, we can identify the $L^2$-spaces on $M_t$ and $\Sigma. $ Since the leading eigenfunction $\phi_1$ of $L_\Sigma$ is strictly positive, this implies that $v(t)$ has a nontrivial projection to the $\phi_1$-direction. See the red curve in Figure \ref{Figure}. For not necessarily positive initial data, we show that either it will drift to the first eigenfunction direction, or it will drift to the first eigenfunction after adding a small positive perturbation. In fact, we prove that drifting to the first eigenfunction direction is related to the growth rate of the function (see Lemma \ref{Lm:growth=proj}), and we can always make the solution grow sufficiently fast by adding a small positive perturbation.
	
	In the second step, we show that if the difference $M_T$ and $\widetilde M_T$, when considered as a function on $\Sigma$, has a nontrivial projection to the $\phi_1$-direction, then the local dynamics will grow the $\phi_1$-component exponentially so that it dominates all the other Fourier modes. See the blue curve in Figure \ref{Figure}. Thus, we get a perturbation towards the $\phi_1$-direction on the limit shrinker by a positive perturbation on the initial condition, and \cite{ColdingMinicozzi12_generic} has proved that such a perturbation will cause entropy decrease. 
	
	\subsection{Invariant manifolds}
	In dynamical systems, for local (in)stability analysis, the standard machinery is to construct the invariant manifolds (stable/unstable/center manifolds) associated to an invariant set. For instance, a stable (respectively unstable) manifold is the set of points that approach the fixed point in the future (respectively past) with an exponential rate. For a generic initial point, its orbit is shadowed by an orbit on the center-unstable manifold, which has finite dimensions. The construction of invariant manifolds in general gives a good classification of the local dynamics near a fixed point. In this paper, we construct invariant manifolds near each compact nonspherical shrinker.  
	
	Let $\Sigma\subset \R^{n+1}$ be a closed embedded smooth $n$-dimensional shrinker. Then the linearized operator $L:=\Delta_\Sigma-\frac{1}{2}\langle x,\nabla\rangle+\frac{1}{2}+|A|^2$ has infinitely many negative eigenvalues, finitely many zero eigenvalues, and finitely many positive eigenvalues. We call the number of positive eigenvalues the Morse index, denoted by $I$. For a sphere, $I=n+2$; for a nonspherical shrinker, $I>n+2$.

	Let $\cX$ be the Banach space of $C^{\infty}$ functions completed in the $C^{2,\al}$ norm for $\al\in(0,1)$. We next introduce a splitting on $\cX$. We denote by $\phi_i$ the eigenfunction of $L$ corresponding to the eigenvalue $\lambda_i,\ i=1,2,\ldots$ and denote $$ \cX^c:=\oplus_{\lambda_i=0} \phi_i \R,\quad \cX^u:=\oplus_{\lambda_i>0} \phi_i \R$$
	We denote by $\Pi_*: \ \cX\to \cX^*,\ *=c,u$ the $L^2$-projection and by $\Pi_s:=\mathrm{Id}-(\Pi_c+\Pi_u)$. We next introduce $\cX^s:=\Pi_s\cX$. Let $v=v_s+v_c+v_u\in \cX$ be the decomposition with respect to the splitting, then we introduce the norm on $\cX$ as $\|v\|_\cX:=\|v_s\|_{C^{2,\al}}+\|v_c\|_{C^{2,\al}}+\|v_u\|_{C^{2,\al}}$.

	Our goal is to show that there exists a stable manifold $\cW^{s}$ in $\cX$ which can be written as a graph over $\cX^s$ near zero consisting of points converging to zero as $t\to\infty$, and similarly there exist an unstable manifold $\cW^u$ as a graph over $\cX^u$ consisting of points converging to zero as $t\to-\infty$ and center manifold $\cW^c$ as a graph over $\cX^c$.  However, in general, the orbits on $\cW^c$ may not converge to zero under forward or backward flows. In particular, in both the future direction and the past direction, these orbits may escape the small neighborhood where the linear approximation to the RMCF is valid. In general, the way to solve this problem is to choose a truncation of the flow such that outside a small neighborhood of 0 in $E$ the nonlinearity vanishes and the dynamics on the center manifold are trivial. The equation for the truncated flow is given later in \eqref{RMCFCpt}, called the TRMCF. The resulting center manifold of the modified system thus depends on the choice of the truncation, which is quite flexible. Meanwhile, in a sufficiently small neighborhood of $0$, the stable and unstable manifolds, as well as those orbits on the center manifold converging to $0$ in the future or the past, do not depend on the truncation.

	In this paper, we have the following theorem on the existence of invariant manifolds.%

	\begin{Thm}\label{ThmInvariantMfdCpt}
		Let $\Sigma^n$ be a smooth, closed embedded shrinker in $\R^{n+1}$. Then there exists a sufficiently small $\dt$ such that in the $\dt$-ball $B_\dt(0)$ of $\cX$, the following hold:
		\begin{enumerate}
			\item There is a Lipschitz manifold $\cW^s$ that is the graph of a function $w^s:\ \cX^s\cap B_{\dt}(0) \to \cX^{c}\oplus \cX^u$, with $w^s(0)=0$, $T_0\cW^s=\cX^s$, and every point on $\cW^s$ has its forward orbit under the RMCF converges to $0$ exponentially.
			\item There is a Lipschitz manifold $\cW^u$ of dimension $I$, that is the graph of a function $w^u:\ \cX^u\cap B_{\dt}(0)\to \cX^{s}\oplus \cX^c$, with $w^u(0)=0$, $T_0\cW^u=\cX^u$, and every point on $\cW^u$ has its backward orbit under the RMCF converges to $0$ exponentially.
			
			\item There is a Lipschitz manifold $\cW^c$ that is invariant under the TRMCF, and is the graph of a function $w^c:\ \cX^c\cap B_{\dt}(0)\to \cX^{s}\oplus \cX^u$, with $w^c(0)=0$, $T_0\cW^c=\cX^c$.
		\end{enumerate}
	\end{Thm} 
	
	\begin{Rk}
		We have the following remarks. 
		\begin{enumerate}
			\item Here, we use the theory of maximal regularity to deal with the loss of two derivatives in the nonlinearity of the equation. The advantage of this theory is that we can work with the $C^{2,\alpha}$-norm. Moreover, the approach does not rely on the detailed information of the nonlinearity. One may also take advantage of the special structure of the nonlinearity and work in a high Sobolev space as in \cite{EpsteinWeinstein87_StableManifold} and later by Strehlke \cite{Strehlke20_Asym_Max}; also see the work of Martin-Hagemayer and Sesum \cite{Martin-HagemayerSesum23_RateRMCF}. 
			\item In Proposition 2.3 of \cite{ColdingMinicozzi19_Dynamics}, Colding-Minicozzi constructed a stable object $W$ whose manifold structure is unknown. The paper \cite{ColdingMinicozzi19_Dynamics} has an important feature that is to modulo the Euclidean affine transformation group, which can also be incorporated into our local analysis here. We refer readers to \cite{ColdingMinicozzi19_Dynamics} for more details. 
			
			\item In \cite{ChoiMantoulidis22AncientGradientFlows}, Choi-Mantoulidis proved that for compact shrinkers, the existence of ancient solutions that form an $I$-parameter family where $I$ is the dimension of positive eigenvalues counting multiplicity. Our theorem recovers their result and proves, in addition, the manifold structure of these ancient solutions in the function space $C^{2,\al}$. 
			
		\end{enumerate}
	\end{Rk}
	
	Local dynamics of singularities of MCF have been studied in other settings. Epstein-Weinstein studied the dynamics of singularities of curve shortening flows ($1$-dimensional MCF in the plane) in \cite{EpsteinWeinstein87_StableManifold}, and they developed a theory of stable/unstable manifolds near a shrinking curve under the RMCF of curves in the plane. The first-named author and Baldauf \cite{baldauf-sun_2018sharp} implemented Colding-Minicozzi's generic perturbation of MCF in the case of immersed curves in the plane.
	\subsection{Organization of the paper}

	The paper is organized as follows.  In Section 2, we set up the dynamics of the RMCF in a small neighborhood of a closed shrinker. In Section 3, we prove the global genericity and study the RMCF with perturbations on initial data. In Section 4, we discuss how we can see the ancient solution arising from the perturbations on initial data. In Section 5, we prove the invariant manifold theorem. We also prove the necessary ingredients of the proof in the Appendices.

	\subsection*{Acknowledgment}
	We would like to thank Professor Tobias Colding and Professor William Minicozzi for stimulating discussions. The work is deeply influenced by their insights. 
	We would like to thank Professor Chongchun Zeng, from whom we learned the theory of maximal regularity. We want to thank Zhihan Wang for the enlightening discussion leading to Theorem \ref{Thm:PosDom}. 
	J.X. is supported by the grant NSFC  No. 12271285 and NSFC No. 12361141812, the New Cornerstone Investigator program and Xiaomi Foundation.

	\section{The setup for a dynamical system}
	In this section, we set up the RMCF equation as a dynamical system in a neighborhood of a closed shrinker. When a hypersurface $M$ is $C^{2,\al}$ close to a shrinker $\Sigma$, we can write it as the graph of a function $u$ over $\Sigma$, i.e. $M=\{x+u(x)\mathbf{n}(x) \ |\ x\in \Sigma\}$. Thus the evolution of $M$ under the RMCF can be written as a PDE of $u$ of the form  $$\partial_t u= L_\Sigma u+\mathcal{Q}(u,\nabla u,\nabla^2 u),$$ where $\cQ$ is given in Section 4 of \cite{ColdingMinicozzi19_Dynamics}. 
	
	The fact that $\cQ$ depends on the first and second order derivatives of $u$ follows from that the mean curvature involves twice differentials. This fact also makes it difficult to write the solution to the PDE as a flow on a fixed Banach space due to the loss of derivatives. \cite{ColdingMinicozzi19_Dynamics,ColdingMinicozzi21_wandering} choose to work in the Sobolev space $H^1(\Sigma)$ and not to care about the manifold structure of the set of functions staying close to zero in the future. In this paper, we use the theory of maximal regularity to overcome the difficulty so that we can talk about the dynamics of the flow given by the PDE on a fixed Banach space that is equivalent to $C^{2,\al}(\Sigma)$. The key point in the theory is that the loss of two derivatives in $\cQ$ can be compensated by the smoothing effect of the heat kernel.  
	
	Thus, in this section, we first introduce the preliminaries of the theory of maximal regularity, including the notion of little H\"older space and how it works for an equation of the above form. We next study the dynamics of the equation in this space and prove the main result, Proposition \ref{LmKeyCpt}. 
	
	To proceed, we choose a truncation as follows. We first pick a $C^\infty$ function $\chi:\ \R\to [0,1]$ such that $\chi(x)=1$ for $|x|\in[0,1]$ and $\chi(x)=0$ for $|x|\geq 2$ and $\chi$ is nonincreasing in $|x|$. We pick a small number $\dt>0$ whose value will be fixed below. 
	In the following, we work with the truncated rescaled mean curvature flow ($\mathrm{TRMCF}$)
	\begin{equation}\label{RMCFCpt}
		\partial_t u=L u+\chi(u/{\dt}) \cQ(u):=Lu+f(u).
	\end{equation}
	We denote by $\Phi^t$ the flow generated by this equation. 
	
	\subsection{The little H\"older space}
	Let $\cX$ be the Banach space of the $C^{2,\alpha}$-closure of $C^\infty$ functions on $\Sigma$. More explicitly, $\cX$ is the space of $C^{2,\al}$ functions satisfying in addition $Lu\in h^{\al}$ for all $u\in \cX$, where \begin{equation}\label{EqLittleHolder}h^\al:=\{u\in C^{\al}\ |\ \lim_{r\to 0_+}\sup_{|x-y|\leq r}r^{-\al}|u(x)-u(y)|=0\}\end{equation} is called the \emph{little H\"older space}. Each $u\in \cX$ gives rise to a hypersurface $\{x+u(x)\mathbf n(x):\ x\in \Sigma\}$.
	
	As we have remarked in the introduction, the presence of the $\nabla^2 u$ term in the nonlinearity makes Duhamel's principle not give a smooth solution to the RMCF equation. We invoke the theory of maximal regularity to remedy this situation by introducing certain interpolation spaces.

	Let $\bar\lambda=\sup\{\lambda\ |\ \lambda\in \sigma(L)\}$ and $\omega$ be a number greater than $\bar\lambda$, and $\mathcal L(X)$ be the space of linear operators densely defined on $X$ equipped with operator norm.
	
	\begin{Lm}\label{LmExp}
		Let    $L$ be the operator $L:=L_\Sigma:\ D\to X$ and the pair $(D,X)$ is each of the following two choices:  $(C^{2,\al}(\Sigma), C^\al(\Sigma))$ and $(h^{2,\al}(\Sigma), h^\al(\Sigma))$, $0<\al<1/2$. 
		Then it generates an analytic semigroup $e^{tL}$ and there exist constants $C_0,C_1$ such that for all $t>0$ 
		$$\|e^{tL}\|_{\mathcal L(X)}\leq C_0 e^{\omega t},\quad \|L e^{tL}\|_{\mathcal L(X)}\leq \frac{C_1}{t}e^{\omega t},$$
	\end{Lm}
	The proof of the last lemma can be found in \cite{Lunardi95_SemiGroup} Theorem 3.1.14 and Corollary 3.1.32 as well as Proposition 2.1.1. 	The following result gives an important characterization of the H\"older and little H\"older spaces in terms of the semigroup $e^{tL}$. 
	\begin{Prop}\label{PropE}Let $\al\in (0,1/2).$
		\begin{enumerate}
			\item The $C^{\al}(\Sigma)$ space is norm equivalent to the space
			$$\cE_{\al/2}:=\{u\in C^0(\Sigma)\ |\ \|u\|_{\al/2}<\infty\}$$
			where the norm $\|\cdot \|_{\al/2}$ is defined to be $\|u\|_{\al/2}:=\|u\|_{C^0}+\sup_{\xi>0} \xi^{1-\al/2}\|Le^{\xi L} u\|_{C^0}$ when $\omega<0$. In case of $\omega>0$, we replace $L$ by $L-2\omega$. 
			\item The $C^{2,\al}(\Sigma)$ space is norm equivalent to the space
			$$\cE_{1+\al/2}:=\{u\in C^2(\Sigma)\ |\ \|Lu\|_{\al/2}<\infty\}$$
			and the norm $\|\cdot \|_{1+\al/2}=\|\cdot \|_{\al/2}+\|L\cdot \|_{\al/2}$. 
			\item The $h^{\al}(\Sigma)$ space is norm equivalent to the closure of $C^2(\Sigma)$ in the $\|\cdot \|_{\al/2}$ norm and the $h^{2,\al}(\Sigma)$ space is norm equivalent to the space of $C^{2}$ functions with $Lu\in h^\al$, $u\in C^{2}$. 
		\end{enumerate}
	\end{Prop}
	The proof of these results can be found in Theorem 2.10 of \cite{Lunardi85_Interpolation} and Chapter 3 (Theorem 3.1.12, 3.1.29, 3.1.30) of \cite{Lunardi95_SemiGroup}.

	We next study the RMCF equation using the theory of maximal regularity. We first formally write down the solution to the RMCF equation 
	\begin{equation}\label{EqDuHamel}
		u(t)=e^{Lt }u_0+\int_0^t e^{L(t-s)}f(u(s))\,ds:=e^{Lt }u_0+V(t,u). 
	\end{equation}
	We have the following crucial proposition, which consists of Propositions 1.1 and 1.2 of \cite{DaPratoLunardi88_Stability}). We include a proof since it is illuminating. 
	\begin{Prop}\label{PropKey}
		\begin{enumerate}
			\item
			Let $f\in C([0,T), h^\al)$, $T\in (0,\infty]$ be such that $\sup_{t\in [0,T)}\|e^{\eta t} f(t)\|_{C^\al}<\infty$. Let $\eta$ be such that $\eta+\omega<0$. Then we have 
			$$\sup_{t\in [0,T)}\|e^{\eta t}\partial_t u(t)\|_{C^\al}+\sup_{t\in [0,T)}\|e^{\eta t}V(t,u)\|_{C^{2,\al}}\leq C\sup_{t\in [0,T)}\|e^{\eta t} f(t)\|_{C^\al}.$$
			\item Suppose $f\in C((-\infty,0], h^\al)$ with $\sup_t e^{\eta t}\|f(t)\|_{C^\al}<\infty$, where $\eta$ satisfies $\eta+\omega<0$. Denote $V_{-\infty}(t,u)=\int_{-\infty}^t e^{L(t-s)}f(u(s))\,ds$. Then we have
			$$\sup_{t\in\R_{\leq 0}}\|e^{\eta t}V_{-\infty}(t,u)\|_{C^{2,\al}}\leq C\sup_{t\in \R_{\leq 0}}\|e^{\eta t} f(t)\|_{C^\al}.$$
		\end{enumerate}
	\end{Prop}
	\begin{Rk}
		The key point of the proposition is that the function $V$, compared with $f$, gains two spatial derivatives. When $f$ depends on $D^2u$, then this proposition implies that we can consider \eqref{EqDuHamel} as an equation in $h^{2,\alpha}$. 
	\end{Rk}
	\begin{proof}
		We only give the proof of item (1), and that of item (2) is completely similar. 
		We consider first the case of $\eta=0$. Using Proposition \ref{PropE}, we have
		\begin{equation*}
			\begin{aligned}
				\|V(t,u)\|_{C^{2,\al}}&\leq C\|V(t,u)\|_{1+\al/2}=C\|LV(t,u)\|_{\al/2}\\
				&=C\sup_\xi\left\|\xi^{1-\al/2}\int_0^t L^2 e^{(\xi+t-s)L}f(u(s))\,ds\right\|_{\al/2}\\
				&=C\sup_\xi\xi^{1-\al/2}\int_0^t \left\|L e^{\frac{(\xi+t-s)}{2}L}L e^{\frac{(\xi+t-s)}{2}L}f(u(s))\right\|_{C^0}\,ds\\
				&\leq C\sup_\xi\xi^{1-\al/2}\int_0^t (\frac{\xi+t-s}{2})^{-1}\|L e^{\frac{(\xi+t-s)}{2}L}f(u(s))\|_{C^0}\,ds\\
				&\leq C\sup_t \|f(t)\|_{\al/2}\cdot\sup_\xi\int_0^t \xi^{1-\al/2}(\frac{\xi+t-s}{2})^{-2+\al/2}\,ds \\
				&\leq C \sup_t \|f(t)\|_{C^\al},
			\end{aligned}
		\end{equation*}
		where in the second $\leq$, we use Lemma \ref{LmExp} and the fact $\omega<0$,  in the third $\leq$, we use the definition of the $\|\cdot\|_{\al/2}$ norm and the fourth $\leq$, we integrate explicitly $$\int_0^t (\frac{\xi+t-s}{2})^{-2+\al/2}\,ds=2^{2-\alpha/2}(1-\alpha/2)(\xi^{-1+\alpha/2}-(\xi+t)^{-1+\alpha/2}).$$ 
		
		We next consider the case of general $\omega$ and $\eta$. We shall substitute $f_\eta(t)=e^{\eta t}f(t)$ and $L_\eta=L+\eta$. Then we reduce the general case to the above case. 
	\end{proof}
	\subsection{Dynamics in a neighborhood of the shrinker}\label{SS:Dynamics in a neighborhood of the shrinker}
	In this section, we study the dynamics of the RMCF in a neighborhood of the shrinker. The key point is that as the $f$ term is quadratically small in $u$, the local dynamics is dominated by the linear part of the equation, i.e., the nonlinear dynamics is well approximated by the linear one $e^{L_\Sigma t}$.
	
	Let $0<\lambda_2<\lambda_1$ be the first two eigenvalues of $L_\Sigma$ and we introduce numbers $\gamma,\beta,\eta,\omega$ such that $$\lambda_2<\gamma<\beta<\lambda_1<\omega<-\eta.$$ This gives rise to a splitting for $\cX=h^{2,\al}$, $
	\cX=X_-\oplus X_+$ where $X_+=\phi_1 \R$, $\phi_1$ is the eigenfunction associated to the leading eigenvalue, $X_-=(I-\Pi_+) \cX$ and $\Pi_+$ is the $L^2$-projection of $\cX$ to $X_+$. The splitting is invariant under $L$ and $e^{Lt}$. We introduce $L_\pm$ as the restriction of $L$ to $X_\pm$ respectively.

	On $X_\pm$, we use the $\|\cdot\|_{1+\al/2}$ norm, and on $\cX$ we use the norm $\|u\|=\|u_-\|_{1+\al/2}+\|u_+\|_{1+\al/2}$ where $u=u_-+u_+$ with respect to the splitting $\cX=X_-\oplus X_+$. By Lemma \ref{LmExp}, we have $\|e^{L_- t} u_-\|_{1+\al/2}\leq C_-e^{\gamma t}\|u_-\|_{1+\al/2}$. We can indeed introduce a norm $\|\cdot\|$ on $X_-$ called Lyapunov norm, that is equivalent to $\|\cdot\|_{1+\al/2}$ such that we have $$\|e^{L_- t} u_-\|\leq e^{\gm t}\|u_-\|$$ for all $u_-\in X_-$, i.e. we can indeed take $C_-=1$. Indeed, it is enough to define $\|u_-\|=\sup_{t\geq 0} e^{-\gamma t}\|e^{L_-t}u_-\|_{1+\al/2}$ (the equivalence is as follows, taking $t=0$, we get $\|u_-\|\geq \|u_-\|_{1+\al/2}$. From $\|e^{L_- t} u_-\|_{1+\al/2}\leq C_-e^{\gamma t}\|u_-\|_{1+\al/2}$, we get $\|u_-\|\leq C_-\|u_-\|_{1+\al/2}$). Similarly, we introduce such a norm on $X_+$ such that $$\|e^{L_+ t}u_+\|\geq e^{\beta t}\|u_+\|$$ for all $t\geq 0$ and $u_+\in X_+$. Then the norm on $\cX$, still denoted by $\|\cdot\|,$ is taken to be the sum of the two new norms on $X_\pm$. With the new norm, we still have a constant $C_1$ such that $$\|e^{L t} u\|\leq C_1e^{\omega t}\|u\|$$ for all $t\geq 0$ and $u\in \cX$.
	
	Let $\kappa$ be a positive number, we define \begin{equation}\label{Eqcone}\mathcal K_\kappa=\{ u=(u_-,u_+)\in X_-\oplus X_+\ |\ \kappa \|u_-\|< \|u_+\|\}.\end{equation}
	The larger $\kappa$ is, the narrower the cone is around $X_+$. 
	
	Denote by $\Phi^1$ the time-1 map defined by equation \eqref{EqDuHamel}. 
	\begin{Prop}\label{LmKeyCpt}
		For all $c>0$, there exists $\dt_0>0$ such that for all $0<\dt<\dt_0$ and all $\kappa\in [c, \bar\kappa], \ \bar\kappa=\frac{1}{2C_1\dt}e^{\eta+\frac\gamma2}(e^\frac\beta2-e^{\frac\gamma2})=O(\dt^{-1}),$ we have the following:
		
		\begin{enumerate}
			\item the cone $\mathcal K_\kappa$ is mapped strictly inside $\mathcal K_{\min\{\bar\kappa,\kappa \sqrt{e^{\beta-\gamma}\}}}$ under $\Phi^1$. Moreover, if $u(0)=(u_-(0),u_+(0))\in \mathcal{K}_\kappa$, then we have
			$$\|u_+(n)\|\geq \|u_+(0)\| e^{n(\beta-O(\dt))},\quad \forall\ n\geq 0,$$
			and $u(n)\in\cK_{\min\{\bar\kappa,\kappa e^{n(\beta-\gamma)/2}\}}$
			provided the piece of orbit $\{u(t),\ t\in [0,n]\}\subset B_{\dt}\cap \cX$.
			\item If $u_2(0)\in u_1(0)+\mathcal{K}_\kappa$, then the corresponding solution satisfies $u_2(n)-u_1(n)\in \cK_{\min\{\bar\kappa,\kappa e^{n(\beta-\gamma)/2}\}}$ and
			$$\|u_{1,+}(n)-u_{2,+}(n)\|\geq \|u_{1,+}(0)-u_{2,+}(0)\| e^{n(\beta-O(\dt))},\quad \forall\ n\ge 0.$$
			provided the piece of orbits $\{u_i(t),\ t\in [0,n]\}\subset B_{\dt}\cap \cX,\ i=1,2$. 
		\end{enumerate}
	\end{Prop}
	\begin{proof}
		We shall work only on the proof of part (1), since part (2) is completely similar. 
		
		Let $u=(u_-,u_+)$. By Proposition \ref{PropKey} and equation \eqref{EqDuHamel}, we get	
		$$\sup_{0\leq t\leq 1}e^{\eta t}\|u(t)-e^{Lt}u(0)\|\leq \dt \sup_{0\leq t\leq 1}e^{\eta t}\|u(t)\|. $$
		Estimating $\|u(t)\| $ by $\|u(t)\| \leq \|u(t)-e^{Lt}u(0)\|+\|e^{Lt}u(0)\|$, we then get
		$$\sup_{0\leq t\leq 1}e^{\eta t}\|u(t)-e^{Lt}u(0)\|\leq \frac{\dt}{1-\dt} \sup_{0\leq t\leq 1}e^{\eta t}\|e^{Lt}u(0)\|.$$
		By the choice of the Lyapunov norm, we get $e^{\eta t}\|e^{Lt}u(0)\|\leq e^{\eta t}e^{\omega t}C_1\|u(0)\|\leq C_1\|u(0)\|$. So we get the estimate for $t\in [0,1]$
		$$\|u(t)-e^{Lt}u(0)\|\leq \frac{\dt C_1}{1-\dt} e^{-\eta }\|u(0)\|.$$
		In particular, we get
		$$\|u_\pm(t)-e^{Lt}u_\pm(0)\|\leq \frac{C_1\dt}{1-\dt} e^{-\eta }\|u(0)\|,$$
		hence
		\begin{equation*}
			\begin{aligned}
				\|u_-(t)\|&\leq \|e^{L_-t}u_-(0)\|+\frac{C_1\dt}{1-\dt} e^{-\eta }\|u(0)\|,\\
				\|u_+(t)\|&\geq \|e^{L_+t}u_+(0)\|-\frac{C_1\dt}{1-\dt} e^{-\eta }\|u(0)\|.\\
			\end{aligned}
		\end{equation*}
		Choosing $u(0)\in \partial \mathcal K_\kappa$, so we have $\kappa \|u_-(0)\|=\|u_+(0)\|$, then we get $$\|u(0)\|=\|u_-(0)\|+\|u_+(0)\|=(1+\kappa)\|u_-(0)\|=(1+\kappa^{-1})\|u_+(0)\|.$$
		So we get the estimate
		\begin{equation*}
			\begin{aligned}
				\|u_-(1)\|&\leq \left(e^{\gamma}+\frac{C_1\dt(1+\kappa)}{1-\dt} e^{-\eta }\right)\|u_-(0)\|,\\
				\|u_+(1)\|&\geq \left(e^{\beta}-\frac{C_1\dt(1+\kappa^{-1})}{1-\dt} e^{-\eta }\right) \|u_+(0)\|.
			\end{aligned}
		\end{equation*}
		Taking the quotient, we get
		$$\frac{\|u_-(1)\|}{\|u_+(1)\|}\leq \frac{e^{\gamma}+\frac{C_1\dt(1+\kappa)}{1-\dt} e^{-\eta }}{e^{\beta}-\frac{C_1\dt(1+\kappa^{-1})}{1-\dt} e^{-\eta }}\frac{\|u_-(0)\|}{\|u_+(0)\|}.$$
		To make sure that the image of $\mathcal K_\kappa$ is mapped strictly inside $\mathcal K_{d\kappa}$ for some $d>1$ by $\Phi^1$, it is enough to require that 
		$$\frac{e^{\gamma}+\frac{C_1\dt(1+\kappa)}{1-\dt} e^{-\eta }}{e^{\beta}-\frac{C_1\dt(1+\kappa^{-1})}{1-\dt} e^{-\eta }}<\frac{1}{d}.$$
		Taking $d=\sqrt{e^{\beta-\gamma}}$, we get the bound on $\kappa$ as $\kappa<\frac{1}{2C_1\dt}e^{\eta+\frac\gamma2}(e^{\beta/2}-e^{\gamma/2})$ for $\dt$ sufficiently small as a sufficient condition for the inequality. 
		
		Moreover, for $\kappa>c$, by choosing $\dt$ sufficiently small, we have 
		$$\|u_+(1)\|\geq \left(e^{\beta}-\frac{C_1\dt(1+c^{-1})}{1-\dt} e^{-\eta }\right) \|u_+(0)\|=e^{\beta}(1-O(\dt)) \|u_+(0)\|.$$
		
	\end{proof}

	\section{The global genericity theorem}\label{S:3}
	In this section, we prove Theorem \ref{ThmGlobalCpt}. We first set up the notations and conventions. 
	Throughout this paper, whenever we discuss the RMCF, $u$ will be the graph function of a graphical RMCF over the limit shrinker $\Sigma$; $u^\star$ will be a solution to the linearized RMCF equation over $\Sigma$; $v$ will be the graph function of a graphical RMCF over the RMCF $M_t$, and $v^\star$ will be a solution to the linearized RMCF equation over $M_t$.
	
	All the Sobolev spaces are defined with respect to the Gaussian area. For example 
	\[\|u\|_{L^2(\Sigma)}:=\left(\int_\Sigma |u(x)|^2 e^{-\frac{|x|^2}{4}} d\mu \right)^{1/2}.\]
	
	We use the convention that $\lambda$ is an eigenvalue of the linearized operator $L$ if there exists a function $f$ such that $Lf=\lambda f$. We want to remind the readers that this convention is different from many contexts in MCF, such as \cite{ColdingMinicozzi12_generic}. We adopt this convention for the convenience of studying dynamics since the eigenvalues can be considered as Lyapunov exponents.

	\subsection{The variational equation and its main properties}
	Let $(M_t)_{t\in[0,T]}$ be a smooth solution to the RMCF $x_t=-(H-\frac{\langle x,\mathbf n\rangle}{2})\mathbf n$ on the time interval $[0,T]$. Let $E(t)$ be the Banach space of $C^{2,\al}$ functions on $M_t,\ t\in[0,T]$. Thus each sufficiently small $v(t) \in E(t)$ gives rise to a hypersurface $N^{v(t)}:=\{x+v(t)\mathbf n_t(x) \ |\ x\in M_t\}$ where $\mathbf n_t$ is the outer normal vector field on $M_t$.
	
	Let $\eps v(0)\in E(0)$, which gives rise to a hypersurface $N^{\eps v(0)}$ that is a graph over $M_0$. Flow $N^{\eps v(0)}$ by the RMCF over time $[0,T]$ and denote by $(N_{\eps,t})_{t\in[0,T]}$ this RMCF. Since there is no singularity of $(M_t)$ on the time interval $[0,T]$, so is the flow $N_{\eps,t}$ if $\eps$ is sufficiently small, by the local well-posedness of the RMCF (c.f. Proposition \ref{PropWP}). Denoting $$N_{\eps,t} = M_t + v_\eps(t)\mathbf n_t ,\quad v_\eps(t)\in E(t),$$ where the difference is taken along the outer normal $\mathbf n_t$ of $M_t$.
	
	We let $\eps\to 0$ and denote $\lim_{\eps\to 0}\frac{v_\eps(t)}{\eps}=v^\star(t)$. Then, to the leading order, we obtain \emph{formally} a linear equation called \emph{variational equation} governing the evolution of $v^\star$. Explicitly, the variational equation reads (See Appendix \ref{SA:Perturbation of the Rescaled MCF} for the derivation). 
	
	\begin{equation*}\partial_t v^\star(t)=L_{M_t} v^\star(t),\end{equation*}
	where $L_{M_t}:=\Delta_{M_t}-\frac12\langle x,\nabla_{M_t}\rangle+(|A|^2+\frac12)$ is the linearized operator associated to the hypersurface $M_t$. This equation is also called {\it linearized RMCF equation} or simply {\it linearized equation}.

	We next study some basic properties of the variational equation. 
	\begin{Prop}\label{PropWP}
		Given an RMCF $M_t$ converging to a closed shrinker $\Sigma$, there exist $\delta_0>0$, $\alpha'>0,\eps_0>0$ and $C$ so that if $v_0:M_T\to\R$, for all $T\geq 0,$ satisfies
		\[
		|v_0|+|\nabla v_0|\leq \delta\leq\delta_0,\quad |\Hess_{v_0}|\leq 1,
		\]
		then there exists a solution $N^{v(t)}$ to the RMCF equation with $ v:\ \cup_{t\in[T,T+\eps_0]}M_t\times\{t\}\to\R$, $v(\cdot,T)=v_0$ and
		\begin{itemize}
			\item $|v|\leq C\delta$, $|\mathrm{Hess}_v|\leq C$, and $|\nabla v|^2\leq C\delta$ on $\cup_{t\in[T,T+\eps_0]}M_t\times\{t\}$.
			\item $\|v(\cdot,T+\eps_0)\|_{C^{2,\alpha'}}\leq C$.
			\item Given $\al\in [0,\al')$, we have $\|v(\cdot, T+\eps_0)\|_{C^{2,\alpha}} \leq C\dt^{\frac{\al'-\al}{2+\al'}}.$
		\end{itemize}
	\end{Prop}
	This proposition is an analogue of Proposition 3.28 of \cite{ColdingMinicozzi19_Dynamics}. In Appendix \ref{SSRMCFGraph}, we will work out the equations of motion for $N_t$ written as a graph of $v$ over a given RMCF $M_t$. After that, the proposition is proved by repeating the argument in Section 3 of \cite{ColdingMinicozzi19_Dynamics}.

	Suppose $v:\ M_t\to \R$ is such that $N^{v(t)}$ is an RMCF close to $M_t$ for some short time, and $v^\star$ satisfies the linearized equation on $M_t$. Define $w=v-v^\star$. Then $w$ satisfies the equation
	\[\partial_t w=L_{M_t}w +\cQ(v).\]
	Notice that $\cQ$ is quadratic from the above proposition and the estimate in \cite{ColdingMinicozzi19_Dynamics} (Lemma 3.5), see also Appendix \ref{SSRMCFGraph}. Then from the parabolic Schauder theorem and previous analysis of the RMCF equation, we get the following proposition characterizing the difference between the RMCF and linearized equation in H\"older norm.
	
	\begin{Cor}\label{Cor:positive time smoothing}
		Let $\{M_t\}_{t\in[0,T]}$ be a smooth RMCF of closed embedded hypersurfaces, and let $0<\tau<T$. There exist $\rho>0$ and $C>0$, depending on the reference flow, $\alpha$, $\tau$, and $T$, with the following property. If $v_0\in C^{2,\alpha}(M_0)$ satisfies $\|v_0\|_{C^{2,\alpha}(M_0)}\leq\rho$, then the RMCF starting from the normal graph of $v_0$ over $M_0$ remains a normal graph of a function $v(\cdot,t)$ over $M_t$ for $t\in[0,T]$, and
		\[
		\sup_{\tau\leq t\leq T}\|v(\cdot,t)\|_{C^{3,\alpha}(M_t)}
		\leq C\|v_0\|_{C^{2,\alpha}(M_0)}.
		\]
	\end{Cor}
	
	\begin{proof}
		With Proposition \ref{PropWP}, the graphical RMCF equation in Appendix \ref{SSRMCFGraph} becomes a uniformly parabolic quasilinear equation
		\[
		\partial_t v=a^{ij}(x,t,v,\nabla v)\nabla_i\nabla_jv+b(x,t,v,\nabla v),
		\]
		whose coefficients are smooth and uniformly controlled. Moreover, $v\equiv0$ is a solution. The local estimates in Proposition~\ref{PropWP}, together with continuous dependence on the initial data, show that
		\[
		\sup_{0\leq t\leq T}\|v(\cdot,t)\|_{C^{2,\alpha}}
		\leq C_T\|v_0\|_{C^{2,\alpha}}
		\]
		when $\|v_0\|_{C^{2,\alpha}}$ is sufficiently small. Fix $\tau>0$. Applying the interior parabolic Schauder estimate on $[\tau/2,T]$, and differentiating the equation once in the spatial variables, gives
		\[
		\sup_{\tau\leq t\leq T}\|v(\cdot,t)\|_{C^{3,\alpha}}
		\leq C_{\tau,T}\sup_{\tau/2\leq t\leq T}\|v(\cdot,t)\|_{C^{2,\alpha}}
		\leq C_{\tau,T}\|v_0\|_{C^{2,\alpha}}.
		\]
		This proves the corollary.
	\end{proof}

	\begin{Prop}\label{Prop:DiffRMCFandLinear}
		Given an RMCF $\{M_t\}_{t\in[0,\infty)}$ converging to a closed shrinker $\Sigma$ and $T>0$, there exist $\delta_0>0$, $\epsilon_0>0$ and $C$ so that, if $v(0):M_0\to\R$ satisfies
		\[
		|v(0)|+|\nabla v(0)|\leq \delta\leq\delta_0,\quad |\Hess_{v(0)}|\leq 1,
		\]
		Then we have for $t\in[1,T]$
		\begin{equation*}
			\|(v-v^\star)(\cdot,t)\|_{C^{2,\alpha}}\leq C\delta^{1+\epsilon_0}.
		\end{equation*}
	\end{Prop}
	\subsection{Li-Yau estimate and Harnack inequality}
	The variational equation is a time-dependent linear parabolic equation. We are interested in its solutions with positive initial conditions. By the maximum principle, the solution remains positive for all time. Finer information about positive solutions is given by the following Li-Yau estimate and Harnack inequality. Since they are of independent interest, we study them systematically in Appendix \ref{SLiYau}. The following is an analogous theorem in our setting.
	
	\begin{Thm}\label{ThmLiYau}
		Let $\{M_t\}_{t\in[0,1]}$ be an RMCF and $v^\star$ be a positive solution to the variational equation $\partial_tv^\star=L_{M_t}v^\star$. Then there exist $\delta>0$ and $C>0$ depending on $(M_t)_{t\in[0,1]}$, such that
		\begin{equation*}
			\partial_t (\log v^\star)\geq \frac{|\nabla \log v^\star|^2}{2}-C-\frac{C}{t}\quad \text{on $(0,4\delta)$}.
		\end{equation*}
	\end{Thm}
	Compared with the standard Li-Yau estimate in \cite{LiYau86_parabolic_kernel}, our Li-Yau estimate deals with the Bakry-Emery Laplacian $\Delta_{M_t}-\langle x,\nabla\rangle$ with a potential $|A|^2+\frac12$. Our result differs from existing results (for instance \cite{Lee17_GeneLiYau, Li15_LYHforBER}) in that our $L_{M_t}$ operator is time-dependent.

	The next ingredient is the following Harnack inequality. In general, the Harnack inequality derived from the Li-Yau estimate can only compare oscillations of solutions on different time slices, and the estimate blows up on the same time slice. The following Harnack inequality allows us to estimate the oscillation of positive solutions on each time slice.
	
	\begin{Thm} \label{Thm:Averaging property of rMCF}
		Let $\{M_t\}_{t\in[0,\infty)}$ be an RMCF converging to a compact shrinker $\Sigma$ smoothly. There exist $\widetilde{t}>0$, $C>0$ such that the following is true: let $v^\star$ be a positive solution to the RMCF equation
		$\partial_t v^\star=L_{M_t} v^\star$
		on the time interval $[0,\infty)$, then we have
		\begin{equation*}
			\max v^\star(\cdot,t)\leq C\min v^\star(\cdot,t),\quad \forall\ t\geq \widetilde t.
		\end{equation*}
	\end{Thm}
	We are interested in the case of $M_t\to \Sigma$ under the RMCF where $\Sigma$ is a compact shrinker. For $t$ sufficiently large, we assume that $M_t$ is sufficiently close to $\Sigma$ in the $C^k$ topology for any $k\geq 4$, so we can write $M_t$ as a graph over $\Sigma$. In fact, even if  only a subsequence of $M_t$ converges to $\Sigma$ in the Lipschitz norm, the uniqueness of tangent flow proved by Schulze in \cite{Schulze14_UniquenessCompactShrinker} shows that $M_t$ converges to $\Sigma$ smoothly, independent of the choice of subsequence. 
	
	Theorem \ref{ThmLiYau} and Theorem \ref{Thm:Averaging property of rMCF} are proved in Appendix \ref{SLiYau}.
	
	At time $t$, a solution to the variational equation is a function on $M_t$. When $t$ is sufficiently large, we would like to write the graph of a function $v$ on $M_t$ as a graph over $\Sigma$. So we introduce the notion of \emph{transplantation}.
	\begin{Def}[Transplantation]
		Let $\Sigma$ be a closed embedded hypersurface and $\Sigma'$ be a graph of a function $f$ over $\Sigma$. Given a function $u$ on $\Sigma'$, we say that $\bar u:\ \Sigma\to \R$ is a \emph{transplantation} of $u$, if $\bar u(x)=u(x+f(x)\bn(x))$. \end{Def}
	We have the following theorem relating graphs on $M_t$ and graphs on the shrinker $\Sigma$.
	\begin{Thm}\label{Thm:closeness of graphs}
		Let $\Sigma$ be a fixed embedded closed hypersurface. Then given $\eps>0$, there exists a constant $\mu(\eps)>0$ with $\mu(\eps)\to 0$ as $\eps\to 0$, such that the following is true: Suppose
		\begin{enumerate}
			\item $\Sigma_1$ is the graph $\{x+f(x)\bn(x):x\in\Sigma\}$ over $\Sigma$,
			\item and $\Sigma_2$ is the graph $\{y+g(y)\bn(y):y\in\Sigma_1\}$ over $\Sigma_1$,
			\item and $\|f\|_{C^4(\Sigma)}\leq \mu $, $\|g\|_{C^{2,\al}(\Sigma_1)}\leq \mu$,
		\end{enumerate}
		then $\Sigma_2$ is a graph of a function $v$ on $\Sigma$,
		and
		\begin{equation*}
			\|v-(f+\bar g)\|_{C^{2,\al}(\Sigma)}\leq \eps\|\bar g\|_{C^{2,\al}(\Sigma)}
		\end{equation*}
		where $\bar{g}(x)=g(x+f(x)\bn(x)) $ is the transplantation of $g$ to $\Sigma$.
		
	\end{Thm}
	This theorem will be proved in Appendix \ref{S:Closeness of graphs}.
	
	Combining all the ingredients above, we can prove that, after a sufficiently long time, the positive solution (after transplantation) to the linearized equation on $M_t$ will be close to the first eigenfunction direction on the limit shrinker. Let $\phi_1$ be the eigenfunction associated to the leading eigenvalue of the operator $L_{\Sigma}$, and we define $\Pi_{\phi_1}$ to be the $L^2$-projection to the space spanned by $\phi_1$. 
	
	\begin{Prop}\label{PropTangent}
		Let $\{M_t\}_{t\in[0,\infty)}$ be an RMCF converging to the compact shrinker $\Sigma$ and $v^\star(t):\ M_t\to \R$ be a positive solution to the equation $\partial_tv^\star=L_{M_t}v^\star$. Then there exist $t'$ sufficiently large and $C>0$ such that we have for all $t>t'$
		\begin{equation*}
			\|\overline{v}^\star(t)-\Pi_{\phi_1}\overline{v}^\star(t)\|_{C^{2,\al}(\Sigma)}\leq C \|\Pi_{\phi_1}\overline{v}^\star(t)\|_{C^{2,\al}(\Sigma)}.
		\end{equation*}
		where $\overline{v}^\star(t)$ is the transplantation of $v^\star(t)$ from $M_t$ to $\Sigma$. 
	\end{Prop}
	This proposition will be proved in Section \ref{SS:Proj1st}. 
	
	\subsection{Projection to the $\phi_1$ component}\label{SS:Proj1st}
	In this section, we are going to prove Proposition \ref{PropTangent}. Because $\|\cdot\|_{\cE^{1+\al}}$ is equivalent to $\|\cdot\|_{C^{2,\alpha}}$ (see \cite[Theorem 3.1.29]{Lunardi95_SemiGroup}), we only need to prove Proposition \ref{PropTangent} for $C^{2,\alpha}$-norm.
	
	Let $\phi_1(t)$ be the first eigenfunction of $L_t$ on $M_t$. Because when $t$ sufficiently large, $M_t$ is close to $\Sigma$, $\phi_1(t)$ have some uniform bound. Then we have the following inequality that holds when $t$ is sufficiently large:
	\[C^{-1}\leq \frac{\|\phi_1(t)\|_{C^{2,\alpha}(M_t)}}{ \|\phi_1(t)\|_{L^2(M_t)}}\leq C\]
	for some $C>0$. In fact, both sides are positive numbers, so such $C$ must exist. Using the equivalence of all norms for finite-dimensional spaces, we have the following estimate:
	\[
	\|\Pi_{\phi_1(t)}f\|_{C^{2,\alpha}(M_t)}\leq C\|\Pi_{\phi_1(t)}f\|_{L^2(M_t)}.
	\]
	In the following, we denote by $\phi_1$, without $t$-dependence, the leading eigenfunction of the operator $L_\Sigma$ on the shrinker $\Sigma$.

	\begin{Lm}\label{Lm:linearizedSolProjphi1}
		Suppose $v^\star>0$ solves the linearized equation $\partial_t v^\star=L_{M_t}v^\star$, then there exist $t'>0$ and  $C>0$, such that for all $t>t'$ we have 
		\begin{equation*}
			\frac{1}{C}\leq \frac{\|\Pi_{\phi_1(t)}v^\star(\cdot,t)\|_{C^{2,\alpha}(M_{t})}}{\|v^\star(\cdot,t)\|_{C^{2,\alpha}(M_{t})}}\leq C.
		\end{equation*}
		
	\end{Lm}
	
	\begin{proof}
		In the following, $C$ varies from line to line, but only depends on $M_t$. By the Li-Yau estimate Theorem \ref{Thm:Appendix Averaging property of rMCF}, there exist $\widetilde t>0, \tau>0$ and $C>0$, such that when $T>\widetilde t$ and  $t\in[T+2\tau,T+3\tau]$, we have
		\[\sup_{M_t} v^\star\leq C\inf_{M_t}v^\star.\]
		As a consequence, 
		\begin{equation*}
			\sup_{t\in[T+2\tau,T+3\tau]} \|v^\star\|_{C^0}\leq C\sup_{t\in[T+2\tau,T+3\tau]}\|v^\star\|_{L^2}.
		\end{equation*}
		Recall the non-standard Schauder estimate of parabolic equation (see \cite{ChodoshSchulze21_ConicalShrinkerUniqueness} Theorem 3.6) implies that
		\[ \sup_{t\in[T+2.5\tau,T+3\tau]}\|v^\star(\cdot,t)\|_{C^{2,\alpha}}\leq C\sup_{t\in[T+2\tau,T+3\tau]}\|v^\star(\cdot,t)\|_{C^0}.\]
		Therefore
		\begin{equation*}
			\sup_{t\in[T+2.5\tau,T+3\tau]}\|v^\star(\cdot,t)\|_{C^{2,\alpha}}
			\leq 
			C\sup_{t\in[T+2\tau,T+3\tau]}\|v^\star\|_{L^2}.
		\end{equation*}
		We also recall that $\sup_{M_t} v^\star\leq C\inf_{M_t}v^\star$ implies that
		\[\|v^\star(\cdot,t)\|_{L^2}\leq C\|\Pi_{\phi_1(t)}v^\star(\cdot,t)\|_{L^2}.\]
		Thus,
		\begin{equation*}
			\begin{split}
				\sup_{t\in[T+2.5\tau,T+3\tau]}\|v^\star(\cdot,t)\|_{C^{2,\alpha}}
				\leq 
				&
				C\sup_{t\in[T+2\tau,T+3\tau]}\|v^\star\|_{L^2}
				\\
				\leq
				& 
				C\|v^\star(\cdot,T+3\tau)\|_{L^2}
				\\
				\leq 
				&
				C\|\Pi_{\phi_1(T+3\tau)}v^\star(\cdot,T+3\tau)\|_{L^2}
				\\
				\leq
				&
				C\|\Pi_{\phi_1(T+3\tau)}v^\star(\cdot,T+3\tau)\|_{C^{2,\alpha}},
			\end{split}
		\end{equation*}
		where the second inequality uses Corollary \ref{Cor: Appendix Harnack to rMCF}. This implies that
		\begin{equation*}
			\|v^\star(\cdot,T+3\tau)\|_{C^{2,\alpha}}
			\leq
			C\|\Pi_{\phi_1(T+3\tau)}v^\star(\cdot,T+3\tau)\|_{C^{2,\alpha}}.
		\end{equation*}
		The other direction is a consequence of 
		\begin{equation*}
			\begin{aligned}
				\|\Pi_{\phi_1(T+3\tau)}v^\star(\cdot,T+3\tau)\|_{C^{2,\alpha}}
				&\leq 
				C\|\Pi_{\phi_1(T+3\tau)}v^\star(\cdot,T+3\tau)\|_{L^2}\\
				&\leq 
				C \|v^\star(\cdot,T+3\tau)\|_{L^2}\\
				&\leq 
				C\|v^\star(\cdot,T+3\tau)\|_{C^{2,\alpha}}.
			\end{aligned}
		\end{equation*}
		The statement follows since $T$ is an arbitrary time later than $\widetilde t.$
	\end{proof}
	
	The following corollary is immediate by applying the triangle inequality. 
	\begin{Cor}\label{Lm:smallness linear solution}
		Under the same assumption as Lemma \ref{Lm:linearizedSolProjphi1}, there exists $t'$ sufficiently large and $C>0$ such that when $t>t'$,
		\begin{equation*}
			\|\Pi_{\phi_1(t)}v^\star(\cdot,t)\|_{C^{2,\alpha}(M_{t})}\geq C\|v^\star(\cdot,t)-\Pi_{\phi_1(t)}v^\star(\cdot,t)\|_{C^{2,\alpha}(M_{t})}.
		\end{equation*}
	\end{Cor}
	Next, we prove Proposition \ref{PropTangent}.
	
	\begin{proof}[Proof of Proposition \ref{PropTangent}]
		We assume that $t$ is sufficiently large such that $M_t$ can be written as a graph of function $f$ on $\Sigma$, with $\|f\|_{C^4(\Sigma)}$ sufficiently small. From Corollary \ref{Lm:smallness linear solution}, when $t$ is sufficiently large, we know that 
		\begin{equation*}
			\|\Pi_{\phi_1(t)}v^\star(\cdot,t)\|_{C^{2,\alpha}(M_t)}\geq C\|v^\star(\cdot,t)-\Pi_{\phi_1(t)}v^\star(\cdot,t)\|_{C^{2,\alpha}(M_t)}.
		\end{equation*}
		Because $M_t$ is sufficiently close to $\Sigma$, we have
		\begin{equation*}
			\|\Pi_{\overline{\phi}_1(t)}\overline{v}^\star(\cdot,t)\|_{C^{2,\alpha}(\Sigma)}\geq C\|\overline{v}^\star(\cdot,t)-\Pi_{\overline{\phi}_1(t)}\overline{v}^\star(\cdot,t)\|_{C^{2,\alpha}(\Sigma)}.
		\end{equation*}
		Thus, we only need to prove that
		\begin{equation*}
			\|\Pi_{\overline{\phi}_1(t)}\overline{v}^\star(\cdot,t)\|_{C^{2,\alpha}(\Sigma)}
			\leq 
			C\|\Pi_{\phi_1}\overline{v}^\star(\cdot,t)\|_{C^{2,\alpha}(\Sigma)},
		\end{equation*}
		and
		\begin{equation*}
			\|\Pi_{\overline{\phi}_1(t)}\overline{v}^\star(\cdot,t)-\Pi_{\phi_1}\overline{v}^\star(\cdot,t)\|_{C^{2,\alpha}(\Sigma)}\leq 
			C\|\Pi_{\phi_1}\overline{v}^\star(\cdot,t)\|_{C^{2,\alpha}(\Sigma)}.
		\end{equation*}
		Since the first inequality implies the second one by the triangle inequality, we only prove the first inequality. 
		
		When $t$ is sufficiently large, we have assumed $f$ whose $C^4$ norm is sufficiently small. Then $\overline{\phi}_1$, the first eigenfunction after transplantation, must converge to $\phi_1$ smoothly by Harnack inequality. In particular, by Harnack inequality, we know that when $t$ is sufficiently large, there must be a constant $C$ such that 
		\[\frac{1}{C}\leq \frac{\overline{\phi}_1(t)}{\phi_1}\leq C,\quad \|\overline{\phi}_1(t)\|_{C^{2,\alpha}}\leq C\|\phi_1\|_{C^{2,\alpha}}.\]

		Notice that $\Pi_{\overline{\phi}_1(t)}\overline{v}^\star(\cdot,t)=\langle \overline{\phi}_1(t),\overline{v}^\star(\cdot,t)\rangle \overline{\phi}_1(t)$ and $\Pi_{{\phi_1}}\overline{v}^\star(\cdot,t)=\langle {\phi_1},\overline{v}^\star(\cdot,t)\rangle{\phi_1}$. Hence,
		\begin{equation*}
			\|\Pi_{\overline{\phi}_1(t)}\overline{v}^\star(\cdot,t)\|_{C^{2,\alpha}}=\langle \overline{\phi}_1(t),\overline{v}^\star(\cdot,t)\rangle \|\overline{\phi}_1(t)\|_{C^{2,\alpha}}\leq C\langle \phi_1,\overline{v}^\star(\cdot,t)\rangle\|\phi_1\|_{C^{2,\alpha}}
			=
			C\|\Pi_{{\phi_1}}\overline{v}^\star(\cdot,t)\|_{C^{2,\alpha}}.
		\end{equation*}
	\end{proof}
	
	\medskip
	
	If the initial function $v^\star(0)$ is not positive, then it may not drift to the first eigenfunction as $t\to\infty$. However, if it adds a small positive function, the positive part will drift to the first eigenfunction direction and finally dominate the whole function. More precisely, we have the following theorem:
	
	\begin{Thm}\label{Thm:PosDom}
		There exists $C>0$ with the following significance. For any $C^{2,\alpha}$ initial condition $v^\star(0)$ and any positive function $w_0$, then for all but at most one $\eps\in\R$ and $t_i\to\infty$ such that the solution to the linearized RMCF $(v^\star)'$ with initial condition $v^\star(0) +\eps w_0$ satisfies 
		$
		(\overline{v}^\star)'(\cdot,t_i)\in\cK_C.
		$\end{Thm} 
	
	The basic idea of the proof is the following: the positive part will drift to the first eigenfunction direction and grow exponentially. Thus, it will dominate the other modes as time increases.
	
	We first show that the growth rate of the linearized RMCF equation is related to the projection to the first eigenfunction direction.
	
	\begin{Lm}\label{Lm:growth=proj}
		there exist $C>0$ and $c>0$ with the following significance. Suppose $v^\star(\cdot,t)$ satisfies the linearized RMCF equation. Then $\lim_{t\to\infty} \frac{1}{t}\log \|\overline{v}^\star(\cdot,t)\|\geq \lambda_1(\Sigma)-c$ if and only if there is a sequence of $t_i\to\infty$ such that 
		\begin{equation*}
			\overline{v}^\star(\cdot,t_i)\in \cK_C.
		\end{equation*}
	\end{Lm}
	
	\begin{proof}
		We only need to consider $t$ sufficiently large when $M_t$ can be written as a graph over $\Sigma$. Then we can write the linearized RMCF equation as an equation over $\Sigma$, after transplanting the solution. More precisely, $\overline{v}^\star$ satisfies the following equation
		\begin{equation}
			\partial_t \overline{v}^\star=L_{\Sigma} \overline{v}^\star+g(x,t,\overline{v}^\star,\nabla \overline{v}^\star,\nabla^2\overline{v}^\star).
		\end{equation}
		where $g$ is linear in $\overline{v}^\star,\nabla \overline{v}^\star,\nabla^2 \overline{v}^\star$. Moreover, $g\to 0$ as $t\to\infty$, because $M_t$ converges to $\Sigma$ smoothly. Then we can repeat all the discussions in Section \ref{S2}, by replacing the nonlinear term $f$ by $g$. In particular, Proposition  \ref{LmKeyCpt} holds for the linearized RMCF equation after transplantation.
		
		The only difference is that, in order to have smallness of the nonlinear term, in Section \ref{S2} and Proposition  \ref{LmKeyCpt}, we need a $\delta$-bound for the size of $\|u\|$. For the linearized RMCF equation after transplantation, the smallness of the ``nonlinear" term depends on time $t$, but does not depend on the neighbourhood size of $\Sigma$. Thus, we can repeat the proof of Proposition  \ref{LmKeyCpt} to see that once $v^\star(\cdot,t)\in\cK_C$ for a sufficiently large $t$, it will lie in the cone at an exponential rate, and the rate can be chosen as close to $\lambda_1$ as possible. In particular, $\lim_{t\to\infty} \frac{1}{t}\log \|\overline{v}^\star(\cdot,t)\|_{L^2(\Sigma)}\geq \lambda_1(\Sigma)-c$.  This shows the ``if" part.
		
		Next, we prove the ``only if" part. We prove by contradiction. Suppose $\overline{v}^\star(\cdot,t)$ grows at the rate at least $\lambda_1-c$, but $\overline{v}^\star(\cdot,t)\not\in\cK_C$ for all $t>T$ for some $T>0$. Then $\|\overline{v}^\star(\cdot,t)\|\leq (1+C)\|\overline{v}^\star_-(\cdot,t)\|$, where $\overline{v}^\star_-(\cdot,t)$ is the projection of $\overline{v}^\star(\cdot,t)$ to $X_-$. Notice that we have the eigenvalue gap (see Section \ref{SS:Dynamics in a neighborhood of the shrinker}), and when $t$ is sufficiently large, $g$ is very small. So $\|\overline{v}^\star_-(\cdot,t)\|$ grows at most at the power $\lambda_1-2c$ when $t$ is sufficiently large, hence $\|\overline{v}^\star(\cdot,t)\|$ grows at most at the power $\lambda_1-2c<\lambda_1-c$. This contradicts the growth rate.
	\end{proof}

	\begin{proof}[Proof of Theorem \ref{Thm:PosDom}]
		We have proved the case for $v^\star>0$. We have two cases. The first case is that $\|\overline{v}^\star\|$ grows at the rate at least $\lambda_1 -c$ as in Lemma \ref{Lm:growth=proj}. Then Lemma \ref{Lm:growth=proj} suggests that $\overline{v}^\star(\cdot,t_i)\in\cK_C$ for a sequence of $t_i\to\infty$. The second case is that $\|\overline{v}^\star\|$ grows at a power less than $\lambda_1-c$. Then for all $\eps>0$, $\eps \overline{w}(t)$, the solution to the linearized equation starting from $\eps w_0$, lies in $\cK_C$ when $t$ is sufficiently large, and by Lemma \ref{Lm:growth=proj}, $\|\eps w(t)\|$ grows at the rate at least $\lambda_1-c$. Thus, $\|(\overline{v}^\star)'\|\geq \|\eps \overline{w}(t)\|-\|\overline{v}^\star\|$ grows at the rate at least $\lambda_1-c$. Then Lemma \ref{Lm:growth=proj} gives the desired result.
	\end{proof}

	\subsection{Proof of Theorem \ref{ThmGlobalCpt}}
	
	In this section, we give the proof of Theorem \ref{ThmGlobalCpt}.
	\begin{proof}[Proof of Theorem \ref{ThmGlobalCpt}]
		Openness in the genericity statement is trivial by the continuous dependence on the initial condition of the RMCF (c.f. Proposition \ref{PropWP}). We only work on the density part. i.e., near a $C^{2,\alpha}$ small perturbation, one can construct a perturbation to realize the statement. We will add a possibly small positive perturbation.
		
		{\bf Step 1, the setup.}
		
		We pick a smooth function $v_0\in C^{2,\alpha}$ on $M_0$ and introduce a perturbation of $M_0$ denoted by $\widetilde M_0:=\{x+ \eta v_0(t,x)\mathbf n_t(x)\ |\ x\in M_t\}$, where $\mathbf n_t$ is the outer normal of $M_t$. Denote by $\widetilde M_t$ the RMCF generated by the initial condition $\widetilde M_0$ and by $v(t):\ M_t\to\R$ such that $\widetilde M_t$ is the normal graph of the function $v(t)$ over $M_t$.

		In the following, we will show that for a sufficiently small $\eta$, the perturbed initial hypersurface $\widetilde M_0$ satisfies the requirements in the statement.

		Suppose $\phi_1$ is the first eigenfunction of $L_\Sigma$ on $\Sigma$. We next introduce a splitting $\cX=X_+\oplus X_-$ as we did in Section \ref{SS:Dynamics in a neighborhood of the shrinker}. 
		By the work of Colding-Minicozzi in \cite{ColdingMinicozzi12_generic}, $\lambda_1>1$ since $\Sigma$ is not a sphere. The eigenvalues corresponding to rigid transformations are all less than or equal to 1, so $X_+$ does not contain generators of rigid transformations.

		{\bf Step 2, fixing the localization time $T$ and cone entering time.}
		
		Let us first consider the case where $v_0>0$.
		
		In the following, we will always assume that $T$ is a fixed time, sufficiently large, such that 
		\begin{enumerate}
			\item $T>\tilde{t}$ in Theorem \ref{Thm:Appendix Averaging property of rMCF} and $T>t'$ in Proposition \ref{PropTangent}. 
			\item for all $t>T$, $M_t$ can be written as a graph of a function $f(t)$ over $\Sigma$ for all future time with
			\[
			\|f(t)\|_{C^4}\leq \min\left\{\frac{\dt}{10},\mu(\dt)\right\},
			\]
			where $\mu(\dt)$ is the constant in Theorem \ref{Thm:closeness of graphs}, with $\eps=\dt$.
		\end{enumerate}
		
		Over the time interval $[0,T]$, we apply Proposition \ref{Prop:DiffRMCFandLinear} to get
		\[
		\|v(\cdot,T)-v^\star(\cdot,T)\|_{C^{2,\alpha}}\leq C\eta^{1+\epsilon},
		\]
		where $v^\star$ solves the variational equation $\partial_t v^\star=L_{M_t}v^\star$ over the time interval $[0,T]$ with the initial condition $v^\star(0)=\eta v_0$. Note that here the choice of $T$ is independent of $v_0$. This verifies the uniformity of $T$ in the statement of the theorem when $v_0>0$.
		
		We next apply Proposition \ref{PropTangent} to get 
		\begin{equation*}
			\|\overline{v}^\star(\cdot,T)-\Pi_{\phi_1}\overline{v}^\star(\cdot,T)\|_{C^{2,\alpha}}
			\leq C \|\Pi_{\phi_1}\overline{v}^\star(\cdot,T)\|_{C^{2,\alpha}(\Sigma)}.
		\end{equation*}
		Thus, by choosing $\eta$ sufficiently small, we get
		\begin{equation*}
			\|\overline{v}(\cdot,T)-\Pi_{\phi_1}\overline{v}(\cdot,T)\|_{C^{2,\alpha}}
			\leq C \|\Pi_{\phi_1}\overline{v}(\cdot,T)\|_{C^{2,\alpha}(\Sigma)}.
		\end{equation*}
		
		We next write $\widetilde M_T$ as a graph of a function $u(T)$ over $\Sigma$. Since $T$ is fixed, the positive-time smoothing estimate gives
		\[
		\|v(\cdot,T)\|_{C^{3,\alpha}(M_T)}
		\leq C_T\eta\|v_0\|_{C^{2,\alpha}(M_0)}.
		\]
		Therefore, after further decreasing $\eta$, we may assume that
		\[
		\|v(\cdot,T)\|_{C^{3,\alpha}(M_T)}\leq\mu(\dt).
		\]
		We can then apply Theorem \ref{Thm:closeness of graphs} to obtain
		\[
		\|u(T)-f(T)-\overline v(T)\|_{C^{2,\alpha}(\Sigma)}
		\leq \dt\|\overline v(T)\|_{C^{2,\alpha}(\Sigma)}.
		\]
		In particular,
		\[
		(1-\dt)\|\overline v(T)\|_{C^{2,\alpha}}
		\leq
		\|u(T)-f(T)\|_{C^{2,\alpha}}
		\leq
		(1+\dt)\|\overline v(T)\|_{C^{2,\alpha}}.
		\]
		Together with the preceding cone estimate for $\overline v(T)$, the stronger additive estimate above implies, by choosing $\dt$ sufficiently small, that
		\[
		u(T)-f(T)\in\mathcal K_\kappa
		\]
		for some $\kappa>0$ independent of $\eta$ and $\dt$.
		
		If $v_0$ is not positive, Theorem \ref{Thm:PosDom} implies that for a positive $C^{2,\alpha}$ function $w_0$ and small $\eps\geq0$ (possibly $\eps=0$), the solution $\overline{v}^\star$ to the linearized RMCF equation with initial value $v_0+\eps w_0$ satisfies
		\[
		\overline{v}^\star(\cdot,t_i)\in\cK_C
		\]
		for a sequence $t_i\to\infty$. We choose $T$ to be one of these $t_i$ satisfying (1), (2) above. After fixing this $T$, we decrease $\eta$ so that the positive-time smoothing estimate gives
		\[
		\|v(\cdot,T)\|_{C^{3,\alpha}(M_T)}\leq\mu(\dt).
		\]
		The argument is then exactly the same as above.
		
		{\bf Step 3, preservation of cone condition. }
		
		Now we only consider the RMCF starting from the above time $T$. By shifting the time, we may assume $T=0$. 
		
		We next apply item (2) of Proposition  \ref{LmKeyCpt}. Fixing $\dt$ sufficiently small with $\dt<\delta_0$ where $\dt_0$ is in  Proposition  \ref{LmKeyCpt} and such that the power $(\beta-O(\delta))$ in Proposition  \ref{LmKeyCpt} is greater than $1$. Then Proposition  \ref{LmKeyCpt}(2) implies that if the orbit $\{u(t),t\in[0,n]\}\subset B_\delta\cap \cX$ and $\kappa e^{n(\beta-\gamma)/2}\leq \bar\kappa$,
		\begin{equation*}
			\|\Pi_{\phi_1}(u(n)-f(n))\|\geq \|\Pi_{\phi_1}(u(0)-f(0))\|e^{n(\lambda_1-O(\dt))}.
		\end{equation*}
		Moreover, since $f(n)$ converges to zero due to the convergence $M_t\to\Sigma$, for sufficiently small $\eta$, we have that $u(n)$ stays in $B_\dt$ for so long a time  that $u(n)-f(n)\in \mathcal K_{\bar \kappa}$ and $\bar\kappa=O(\dt^{-1})$, in other words, the cone stabilizes at  $\mathcal K_{\bar \kappa}$. We define $T'$ as the exit time that is the time when $\|u(T')\|\approx\dt$. This exit time exists due to the exponential growth of $\Pi_{\phi_1}(u(n)-f(n))$ and the decay of $f(n)$ and the stabilization of the cone where $u(n)-f(n)$ lies.

		Now we rephrase the consequence of Step 3: for sufficiently small initial perturbation on $M_0$, there exists $T'>0$, such that the perturbed RMCF $\widetilde{M}_{T'}$ can be written as a graph of function $u(T')$ over $\Sigma$, with $\|u(T')\|\approx \delta$, $u(T')\in \cK_{\bar \kappa}$.

		{\bf Step 4, entropy decrease.}
		
		The first three steps prove most of the statement in Theorem \ref{ThmGlobalCpt}, besides the entropy argument. Namely, we want to show that $\widetilde{M}_{T'}$ we obtained in Step 3 has entropy less than $\lambda(\Sigma)$ by a definite amount.
		
		Denote by $u:\ \Sigma\to \R$ be such that $\widetilde{M}_{T'}$ is written as a normal graph over $\Sigma$ and suppose $u$ has the Fourier expansion $u=\sum_{i\geq 1} \hat u_i\phi_i$, where $\phi_i$ is the eigenfunction associated to the $i$-th eigenvalue and normalized to have $L^2$-norm $1$. Then we have, by the second variation formula, that
		$$\cF(\widetilde{M}_{T'})=\cF(\Sigma)-\sum \lambda_{i} \hat u_i^2+o(\|u\|_{C^{2,\al}}^2).$$
		
		By Step 3, we have $\|u_+\|_{C^{2,\al}}\geq C\delta^{-1}\|u_-\|_{C^{2,\al}}$. Since $u_+=\hat u_1\phi_1$ and $u_-=\sum_{i\geq 2}\hat u_i\phi_i$, we have $\|u_+\|_{W^{1,2}}\geq C\|u_+\|_{C^{2,\al}}$ and $\|u_-\|_{W^{1,2}}\leq C\|u_-\|_{C^{2,\al}}$. This gives $\|u_+\|_{W^{1,2}}\geq C\delta^{-1}\|u_-\|_{W^{1,2}}$. This gives that 
		$$\cF(\widetilde{M}_{T'})\leq \cF(\Sigma)-(1-C\delta^{-1})\lambda_1\|u_+\|^2_{L^2}<\cF(\Sigma)-\frac{\lambda_1}{2}\dt^2.$$
		
		We will use the following observation by Colding-Minicozzi in \cite{ColdingMinicozzi12_generic}. Suppose $\Sigma$ is a closed shrinker, and $\phi_1>0$ is the first eigenfunction of $L_\Sigma$ on $\Sigma$. Colding-Minicozzi introduced a notion called $F$-unstable (see Section 4.3 of \cite{ColdingMinicozzi12_generic}), and they proved that $F$-unstable implies the entropy stable. More precisely, if $\Sigma$ is $F$-unstable in the direction of a function $f$, then $\lambda(\Sigma+sf\bn)<\lambda(\Sigma)$ for sufficiently small $s$. Moreover, if $\Sigma$ is not a sphere, the first eigenfunction $\phi_1$ is $F$-unstable direction, see Section 6 of \cite{ColdingMinicozzi12_generic}. 
		Thus, $\lambda(\widetilde{M}_{t})<\lambda(\Sigma)$ for some sufficiently large $t$. This concludes the whole proof of Theorem \ref{ThmGlobalCpt}.
	\end{proof}
	
	As a consequence of Theorem \ref{ThmGlobalCpt}, we prove  Corollary \ref{CorR3}. This can be viewed as a baby version of Conjecture 8.2 in \cite{CMP15_MCF}.
	\begin{proof}[Proof of Corollary \ref{CorR3}]
		We use a compactness theorem for embedded shrinkers by Colding-Minicozzi in \cite{CM12_compactness}. Let $\{M_t\}$ be an MCF, and the first singularity is $(0,0)$, modeled by a multiplicity $1$ closed shrinker $\Sigma$, which is not a sphere. Otherwise, $\{M_t\}$ itself satisfies the requirement. From Theorem \ref{ThmGlobalCpt}, we know that after a small initial positive perturbation, a perturbed MCF can only have singularities modeled by shrinkers with entropy strictly less than $\lambda(\Sigma)$.  
		
		Now let us argue by contradiction. Suppose after a small perturbation $\eps u_0$ on initial data, for $\eps\in[0,\eps_0]$, the perturbed flow $\{M^\eps_t\}$ only generates a singularity at $(x_\eps,t_\eps)$ modeled by a closed shrinker $\Sigma^\eps$ which is not the sphere. Note that because the singularity is modeled by a closed shrinker, $M^\eps_t$ must converge to a single point as $t\to t_\eps$. Hence $t_\eps>0$. By the lower semi-continuity of Gaussian density for MCF spacetime, $(x_\eps,t_\eps)\to (0,0)$ as $\eps\to 0$. By the compactness of shrinkers in \cite{CM12_compactness}, a subsequence of $\Sigma_\eps$ smoothly converges to a shrinker $\Sigma'$.
		
		If $\Sigma'$ is compact, and $\lambda(\Sigma')=\lambda(\Sigma)$, then by the Lojasiewicz-Simon inequality (see \cite{Schulze14_UniquenessCompactShrinker}), $\lambda(\Sigma^\eps)=\lambda(\Sigma)$ when $\eps$ is sufficiently small, which is a contradiction. If $\Sigma'$ is non-compact, then the diameter of $\Sigma_\eps\to\infty$ as $\eps\to 0$. This implies that $M_t^\eps$ has a diameter lower bound at $t=0$, which contradicts the fact that $M_t$ converges to $(0,0)$ as $t\to 0$.
		
		Therefore, $\Sigma'$ must be compact, have a diameter bound, and $\lambda(\Sigma')<\lambda(\Sigma)$. This implies that $\Sigma^\eps$ has a uniform diameter bound. By Colding-Minicozzi's compactness of shrinkers and Lojasiewicz-Simon inequality, $\Sigma^\eps$ has only finitely many possible entropy values. Thus, after finitely many small positive perturbations on initial data, one can find a $\Sigma^\eps=S^n$. This is a contradiction.
	\end{proof}

	\section{Most unstable orbit and ancient solution}\label{S:Most unstable orbit and ancient solution}
	The first eigenfunction on $\Sigma$ induces the most unstable orbit in the unstable manifold. In \cite{ChoiMantoulidis22AncientGradientFlows}, Choi-Mantoulidis provided a general construction of ancient solutions coming out from a stationary point of an elliptic integrand, in the unstable direction. In particular, their construction can be used to construct an ancient RMCF coming out of a shrinker in the first eigenfunction direction. In this section, we present that such a solution can also be seen from the initial perturbation of RMCF.
	
	In the following, we fix an RMCF $\{M_t\}$ converging to $\Sigma$ smoothly as $t\to\infty$.
	
	\begin{Thm}\label{ThmAncient}
		Suppose $u^i_0$ is a sequence of $C^{2,\alpha}$ functions on $M_0$, with $\|u^i_0\|_{C^{2,\alpha}(M_0)}\to 0$ as $i\to\infty$. Let $\widetilde{M}_t^i$ be the corresponding RMCF starting from $M_0+u_0^i\bn$. Then there exists a sequence of $T_i\to \infty$, such that 
		\[\widetilde{M}_{t-T_i}^i\to \widetilde{M}^\infty_t\]
		subsequentially smoothly on any compact subset of $\R^{n+1}\times (-\infty,0]$, and $\widetilde{M}_t^\infty$ is an ancient RMCF, which is a graph of function $u^\infty(t)$ over $\Sigma$, and $u^\infty(t)>0$ for all $t\in(-\infty,0]$.
	\end{Thm}
	
	The idea is as follows: if the initial perturbation becomes smaller and smaller, the perturbed RMCF will be closer and closer to the orbit induced by the first eigenfunction direction when it leaves $B_\delta(0)\cap\cX$. Therefore, after passing to a limit, the perturbed RMCF will converge to the orbit induced by the first eigenfunction.

	\begin{proof}[Proof of Theorem \ref{ThmAncient}]
		Let us fix $\delta,\delta_2$ as in Step 4 in the proof of Theorem \ref{ThmGlobalCpt}. For any $u_0^i$, with small $C^{2,\alpha}(M_0)$ norm, we follow Step 3 in the proof of Theorem \ref{ThmGlobalCpt} to find time $T_i$ (which is $T''$ there), such that $\widetilde{M}_t^i$ can be written as a graph of function $u^i(t)$ on $\Sigma$ for $t\in[T,T_i]$, $\|u^i(T_i)\|\approx\delta$, and $u^i(t)\in\cK_{\bar \kappa}$ where $\bar k$ depends on $\delta$ fixed in Step 3 in the proof of Theorem \ref{ThmGlobalCpt}, and when $\delta$ become smaller, $\bar\kappa$ becomes larger.
		
		First of all, we notice that $T_i-T\to\infty$ as $i\to\infty$. In fact, if $u_0^i$ has a small $C^{2,\alpha}$ norm, then it takes longer for the perturbed RMCF growth to have a certain distance from $M_t$, therefore it takes longer for the perturbed RMCF to become the graph over $\Sigma$ with a fixed amount of $C^{2,\alpha}$ norm.
		By precompactness of $C^{2,\alpha-\epsilon}$ in $C^{2,\alpha}$, we know that this means $u^i(t-T_i)$ subsequentially converges to a limit $u^\infty(t)$ on any compact subset of $\R^{n+1}\times(-\infty,0]$, in $C^{2,\alpha-\epsilon}(\Sigma)$. Further, using the regularity theory of MCF (see \cite{White05_MCFReg}), we know that the convergence is smooth. Thus, we get a limit RMCF $\{\widetilde{M}_t^\infty\}_{t\in(-\infty,0]}$ and \[\widetilde{M}_{t-T_i}^i\to \widetilde{M}^\infty_t\]
		smoothly.
		It remains to prove $u^\infty_t>0$ for $t\in(-\infty,0]$. Because $\phi_1>0$ on $\Sigma$, there exists a constant $K>0$ such that when $u\in\cK_K$, $u$ is also positive on $\Sigma$. When $u_0$ is smaller, we can pick smaller $\delta'$ and $\delta_2'$, to conclude that $u^i_t$ gets into $\cK_K$ earlier, when $\|u^i_t\|\approx \delta'$. Thus, the limit $u_t^\infty$ is also in $\cK_K$, and it can not be $0$. Hence $u_t^\infty$ is positive for all $t\in(-\infty,0]$.
	\end{proof}

	\begin{Rk}
		This kind of ancient solution is called ``one-sided" in \cite{CCMS20_GenericMCF}. From the uniqueness of the one-sided ancient RMCF theorem proved in \cite{CCMS20_GenericMCF}, we know that this limit ancient flow is unique. Hence, we do not have to take a subsequential limit.
	\end{Rk}

	\section{The invariant manifold theorem}\label{S2}
	
	In this section, we prove the existence of various invariant manifolds near a closed shrinker. First, we verify the regularity of the nonlinear term in \eqref{RMCFCpt}. 
	
	\begin{Lm}\label{LmDifference}The function $f:\ h^{2,\al}\to h^{\al}$ defined by $ f(u):=F(x, u,\nabla u, \nabla^2 u)=\chi(u/\dt) \cQ(u)$ is uniformly Lipschitz: $\frac{\|f(u)-f(v)\|_{C^{\al}}}{\|u-v\|_{C^{2,\al}}}\leq C$. Moreover, the Lipschitz constant approaches zero as $\dt\to 0$. 
		
	\end{Lm}
	\begin{proof}
		The explicit expression of $F$ is provided in \cite{ColdingMinicozzi19_Dynamics} as well as \cite{CCMS20_GenericMCF}. Indeed, by the fundamental theorem of calculus, we have
		\begin{equation*}
			\begin{aligned}
				f(u)-f(v)=&\left(\int_0^1 \partial_2F(x,t\mathbf u+(1-t)\mathbf v) dt\right)(u-v)+\left(\int_0^1 \partial_{3}F(x,t\mathbf u+(1-t)\mathbf v) dt\right)(\nabla u-\nabla v)
				\\&+\left(\int_0^1 \partial_{4}F(x,t\mathbf u+(1-t)\mathbf v) dt\right)(\nabla^2u-\nabla^2v) 
			\end{aligned}
		\end{equation*}
		where $\partial_i$ means the partial derivative with respect to the $i$-th variable of $F$, $i=2,3,4$ and $\mathbf u=(u,\nabla u, \nabla^2 u),\ \mathbf v=(v,\nabla v, \nabla^2 v)$. The statement of the lemma follows from the fact that $\partial_i F: h^{2,\al}\to h^\al$ is smooth. 
	\end{proof}
	We are now ready to give the proof of Theorem \ref{ThmInvariantMfdCpt}. 
	
	With the setup of the little H\"older space, the proof of the invariant manifold theorem is now standard as in literature (see for instance \cite{NakanishiSchlag11_InvariantManifold,DaPratoLunardi88_Stability}). We present details in the following. The proof consists of the following three parts: 
	\begin{enumerate}
		\item  construction of the stable manifold $\cW^s$, 
		\item the unstable manifold $\cW^s$, 
		
		\item the center manifold $\cW^c$. 
	\end{enumerate}Among the three parts, most of the key ideas lie in the first part, and all other parts are obtained by slightly adapting part (1). 
	
	\subsection{The stable manifold $\cW^s$}
	Let $\cX$ with the splitting $\cX=\cX^s\oplus \cX^c\oplus \cX^u$ be as in the statement. Let us further denote by $X_1=\cX^s$ and $X_2=\cX^c\oplus \cX^u$, and denote by $\Pi_i$ the $L^2$ projection $\cX$ to $X_i$, $i=1,2$. Instead of considering the operator $L$, we perturb it to $L+\eps$ for some small $\eps>0$. The perturbation is to make sure that $L+\eps|_{X_2}$ has positive spectrum and $L+\eps|_{X_1}$ has negative spectrum, which gives certain expansion on $X_2$ and contraction on $X_1$ for the semigroups. We will fix $\eps$ in the following. We next denote by $L_i$ the restriction of $L+\eps $ to $X_i$ when $i=1,2$. So the eigenvalues of $L_1$ are all negative, and those of $L_2$ are all positive. We introduce $\omega_1<0$ and $\omega_2>0$ such that $\omega_1>\sup\{\sigma(L_1)\}$ and $\omega_2>\sup\{\sigma(L_2)\}$. 
	
	Note that $X_2$ is a finite-dimensional space and $L_2$ is a finite-dimensional linear operator, so we have for $u\in X_2$ 
	\begin{equation}\label{EqL2}
		\begin{aligned}
			\|e^{L_2 t}u\|_{1+\al/2}&=\|Le^{L_2 t}u\|_{\al/2}=\|L_2e^{L_2 t}u\|_{\al/2}=\sup_\xi \xi^{1-\al} \|L_2^2 e^{\xi L_2}e^{t L_2} u\|_{C^0}\\
			&\leq C e^{\omega_2 t}\sup_\xi \xi^{1-\al} \|L_2 e^{\xi L_2} u\|_{C^0}= C e^{\omega_2 t} \|u\|_{\al/2}.
		\end{aligned}
	\end{equation}
	
	Let $u_1(0)\in X_1$ with $\|u_1(0)\|_{1+\al/2}<d$ for some sufficiently small $d$ to be determined later. We next introduce the operator 
	\begin{equation*}
		\begin{aligned}
			\Lambda u&=e^{L_1 t} u_1(0)+\int_0^t e^{(t-s) L_1}\Pi_1(e^{\eps s}f(e^{-\eps s}u(s)))\,ds\\
			&-\int_t^\infty e^{(t-s) L_2}\Pi_2(e^{\eps s}f(e^{-\eps s}u(s)))\,ds.
		\end{aligned}
	\end{equation*}
	We next show that 
	\begin{Prop}\label{PropContraction}
		For small $a>0,$ the operator $\Lambda:\ \ Y_a\to Y_a$ is a contraction hence has a fixed point, where $Y_a$ is defined as the set
		$$Y_a=\{z\in C(\R_{\geq0}, h^{2,\al})\ |\ \sup_{t\geq0} e^{\eta t}\|z(t)\|_{1+\al/2}<a\},$$
		where $\eta$ is a small positive number such that $\omega_1+\eta<0$. 
	\end{Prop}
	It is straightforward to check that a fixed point of $\Lambda$ is a solution to the differential equation $\partial_t u=(L+\eps)u+e^{\eps t}f(e^{-\eps t}u)$, hence $v=e^{-\eps t}u$ is a solution of the equation $\partial_t v=Lv+f(v)$ with the same initial condition $u(0)=v(0)$.
	
	We first show that 
	\begin{Lm}
		$\Lambda$ is a map from $Y_a$ to $Y_a$ if we choose $d$ and $a$ sufficiently small.     
	\end{Lm}
	\begin{proof}
		By Lemma \ref{LmDifference} and by choosing $d$ small, we suppose $f:\ h^{2,\al}\cap B_d\to h^\al$ is $\dt$-Lipschitz where $B_d$ is the $d$-ball in $h^{2,\al}$. 
		
		Take $u\in Y_a$, then it is clear that $e^{-\eps t}u(t)\in h^{2,\al}$ for all $t\geq 0$ by the definition of $Y_a$. We have \begin{equation*}
			\begin{aligned}
				&\|e^{\eta t}\Lambda u(t)\|_{1+\al}\leq C_0\|u_1(0)\|_{1+\al/2}+\sup_t\|e^{\eta t}\Pi_1e^{\eps t}f(e^{-\eps t}u(t))\|_{\al/2}\\
				&+\int_t^\infty Ce^{-\eps (t-s)}e^{\eta s}\|\Pi_2 e^{\eps s}f(e^{-\eps s}u(s))\|_{\al/2}\,ds\\
				&\leq C_0\|u_1(0)\|_{1+\al/2}+\dt\sup_t\|e^{\eta t}u(t)\|_{1+\al/2}+\frac{C\dt}{\eps}\sup_t\|e^{\eta t}u(t)\|_{1+\al/2}\\
				&\leq C_0d+\frac{C\dt}{\eps} a.
			\end{aligned}
		\end{equation*}
		where in the first $\leq$ we use Lemma \ref{LmExp}, Proposition \ref{PropKey} and \eqref{EqL2}, and in the second $\leq$, we use Lemma \ref{LmDifference}. For given $\eps$, we choose $\dt$ and $d$ sufficiently small to get that $\Lambda u(t)\in Y_a$. 
		
	\end{proof}
	
	We next prove the contraction. 
	\begin{proof}[Proof of Proposition \ref{PropContraction}]
		Taking $u,v\in Y_a$ with the same initial condition $u_1(0)$ gives that
		\begin{equation}\label{EqContraction}
			\begin{aligned}
				\|e^{\eta t}\Lambda u(t)-&e^{\eta t}\Lambda v(t)\|_{1+\al/2}\\
				\leq&\|e^{\eta t}\int_0^t e^{(t-s) L_1}\Pi_1(e^{\eps s}f(e^{-\eps s}u(s))-e^{\eps s}f(e^{-\eps s}v(s)))\,ds\|_{1+\al/2}\\
				&+e^{\eta t}\|\int_t^\infty e^{(t-s) L_2}\Pi_2(e^{\eps s}f(e^{-\eps s}u(s))-e^{\eps s}f(e^{-\eps s}v(s)))\,ds\|_{1+\al/2}\\
				\leq &C\sup_t\|e^{\eta t}\Pi_1e^{\eps s}f(e^{-\eps s}u(t))-\Pi_1e^{\eps s}f(e^{-\eps s}v(t))\|_{\al/2}\\
				&+C\int_t^\infty e^{\eta s}e^{(t-s) \eps }\|\Pi_2(e^{\eps s}f(e^{-\eps s}u(s))-e^{\eps s}f(e^{-\eps s}v(s)))\|_{\al/2}\,ds\\
				\leq & C_\eps \sup_t\|e^{\eta t}e^{\eps s}f(e^{-\eps s}u(t))-e^{\eps s}f(e^{-\eps s}v(t))\|_{\al/2}\\
				\leq & C_\eps \dt \sup_t\|e^{\eta t}(u(t)-v(t))\|_{1+\al/2}
			\end{aligned}
		\end{equation}
		where in the second $\leq$, we use Proposition \ref{PropKey} and \eqref{EqL2}, and in the fourth $\leq$, we use Lemma \ref{LmDifference}.
	\end{proof}
	The contraction implies that there exists a fixed point $u=\Lambda u$ with initial condition $u(0)=u_1(0)+\int_0^\infty e^{-s L_2}\Pi_2(e^{\eps s}f(e^{-\eps s}u(s)))ds$. We define the map $$w^s: u_1(0)\mapsto \int_0^\infty e^{-s L_2}\Pi_2(e^{\eps s}f(e^{-\eps s}u(s))).$$ Repeating the calculation in \eqref{EqContraction}, we see that the map $u_1(0)\mapsto u(t)$ going from $X_1\cap B_d$ to $Y_a$ is Lipschitz, i.e. we have $$(1-C\dt) \sup_t\|e^{\eta t}(u(t)-v(t))\|_{1+\al/2}\leq C\|u_1(0)-v_1(0)\|_{1+\al/2}.$$ Then using the expression of $w^s$, we get $$\|w^s(u_1(0))-w^s(v_1(0))\|_{1+\al/2}\leq C_\eps \dt \sup_t\|e^{\eta t}(u(t)-v(t))\|_{1+\al/2}\leq C_\eps \dt \|u_1(0)-v_1(0)\|_{1+\al/2}.$$
	This means that $w^s$ is a $C\dt$-Lipschitz function and, in particular, is differentiable at zero. Then we realize the stable manifold $\cW^s$ as the graph of $w^s$ over $X_1\cap B_d$. This proves part (1) of the theorem. 
	
	\subsection{The unstable manifold $\cW^u$ and the center manifold $\cW^c$  }
	
	\subsubsection{The unstable manifold}
	We consider instead the operator $L-2\eps$ and introduce the spaces $X_i$ and $L_i$ according to the signs of the spectrum of $L-2\eps $ similar to above. Taking $u_2(0)\in X_2$ with norm less than $d$, we introduce the following operator
	
	\begin{equation*}
		\begin{aligned}
			\Lambda u&=e^{L_2 t} u_2(0)+\int_0^t e^{(t-s) L_2}\Pi_2(e^{-\eps s}f(e^{\eps s}u(s)))\,ds +\int^t_{-\infty} e^{(t-s) L_1}\Pi_1(e^{-\eps s}f(e^{\eps s}u(s)))\,ds
		\end{aligned}
	\end{equation*}
	and the space $$Y_a=\{z\in C(\R_{\leq0}, h^{2,\al})\ |\ \sup_{t\leq 0} e^{-\eta t}\|z(t)\|_{1+\al/2}<a\}$$
	where $\eta>0$ is such that $\eta-\bar\omega_2<0,\ \bar\omega_2=\sup\{\sigma(-L_2)\}. $
	We can then repeat the above argument to show that $\Lambda$ has a fixed point in $Y_a$. Hence, the unstable manifold is constructed similarly. This proves part (2).
	
	\subsubsection{ The center manifold}
	
	We next work on the center manifold. We use the splitting $\cX=\cX^s\oplus\cX^c\oplus \cX^u$ and introduce the projections $\Pi_*:\ \cX\to \cX^*$ and the operators $L_*=L|_{\cX^*},\ *=s,c,u.$ Let $u_c(0)\in \cX^c$ with norm $\|u_c(0)\|_{1+\al/2}<d$ for some small $d$ to be determined later. We introduce the operator:
	\begin{equation*}
		\begin{aligned}
			\Lambda_c u&= u_c(0)+\int_{-\infty}^t e^{(t-s) L_s}\Pi_s(f(u(s)))\,ds+\int_0^t \Pi_c(f(u(s)))\,ds -\int_t^{\infty} e^{(t-s) L_u}\Pi_u(f(u(s)))\,ds
		\end{aligned}
	\end{equation*}
	defined on the space of slowly growing functions in both the future and the past
	$$Y_a=\{z\in C(\R, h^{2,\al})\ |\ \sup_{t\in \R} e^{-\eta |t|}\|z(t)\|_{1+\al/2}<a\}$$ 
	where $\eta>0$ is a small number. 
	We first show that $\Lambda_c$ maps $Y_a$ to $Y_a$. For simplicity, we do the estimate for $t\geq0$. To do the similar argument for $t\leq 0$, it is enough to change $\eta$ to $-\eta$
	\begin{equation*}
		\begin{aligned}
			\|e^{-\eta t}\Lambda_c & u\|_{1+\al/2}\\
			\leq &\| e^{-\eta t}u_c(0)\|_{1+\al/2}+\left\|e^{-\eta t}\int_{-\infty}^t e^{(t-s) L_s}\Pi_s(f(u(s)))\,ds\right\|_{1+\al/2}\\
			&+\left\|e^{-\eta t}\int_0^t \Pi_c(f(u(s)))\,ds\right\|_{1+\al/2}+\left\|e^{-\eta t}\int_t^{\infty} e^{(t-s) L_u}\Pi_u(f(u(s)))\,ds\right\|_{1+\al/2}\\
			\leq &C\| u_c(0)\|_{1+\al}+\left\|e^{-\eta t}\Pi_s(f(u(t)))\right\|_{\al/2}+\left\|e^{-\eta t}\int_0^t \Pi_c(f(u(s)))\,ds\right\|_{\al/2}\\
			&+e^{-\eta t}\int_t^{\infty} e^{(t-s) \omega_2}\|\Pi_u(f(u(s)))\|_{\al/2}\,ds\\
			\leq & C\| u_c(0)\|_{1+\al}+\dt\left\|e^{-\eta t}u(t)\right\|_{1+\al/2}+e^{-\eta t}\int_0^t \dt e^{\eta s}\sup_s(e^{-\eta s}\left\|u(s)\right\|_{1+\al/2})\,ds\\
			&+e^{-\eta t}\int_t^{\infty} e^{(t-s) \omega_2}e^{\eta s}\dt \sup_s (e^{-\eta s}\|u(s)\|_{1+\al/2})\,ds\\
			\leq & C\| u_0(0)\|_{1+\al/2}+C\dt \sup_s (e^{-\eta s}\|u(s)\|_{1+\al/2}).
		\end{aligned}
	\end{equation*}
	By choosing $a$ and $\| u_0(0)\|_{1+\al/2}$ small, we prove that $\Lambda_c$ maps $Y_a$ to $Y_a$.
	Similarly, we show that $\Lambda_c:\ Y_a\to Y_a$ is a contraction. The contraction yields a fixed point. The center manifold can be constructed similarly to the stable manifold. This gives part (3).

	\appendix
	\section{Perturbation of the RMCF}\label{SA:Perturbation of the Rescaled MCF}
	
	In this Appendix, we first derive the variational equation and then prove some preliminary properties of it. The first part is a generalization of the analysis in \cite[Appendix A]{CM15_Lojasiewicz, ColdingMinicozzi19_Dynamics, ColdingMinicozzi21_wandering} with the main difference being that in \cite[Appendix A]{CM15_Lojasiewicz, ColdingMinicozzi19_Dynamics, ColdingMinicozzi21_wandering}, Colding-Minicozzi study the RMCF as graphs over a fixed shrinker, which can be viewed as the local behavior of a gradient flow near the fixed point, while we study the behavior of a gradient flow near another gradient flow, so we need to study the RMCF as graphs over another RMCF.
	
	Let us first state the settings. Given a hypersurface $M\subset\R^{n+1}$, let $u$ be a function over $M$, then its graph $M_u$ is defined to be
	\begin{equation*}
		M_u=\{x+u(x)\bn(x)\ |\ x\in M\}.
	\end{equation*}
	
	If $|u|$ is sufficiently small, then $M_u$ is contained in a tubular neighbourhood of $M$ where the exponential map is invertible. Let $e_{n+1}$ be the gradient of the signed distance function to $M$, normalized so that $e_{n+1}$ equals $\bn$ on $M$. We define the following quantities:
	\begin{itemize}
		\item The relative area element $\nu_u(p)=\sqrt{\det g_{ij}^u(p)}/\sqrt{\det g_{ij}(p)},\ p\in M$, where $g_{ij}(p)$ is the metric for $M$ at $p$ and $g_{ij}^u(p)$ is the pull-back metric from $M_u$.
		\item The mean curvature $H_u(p)$ of $M_u$ at $(p+u(p)\mathbf n(p))$.
		\item The support function $\eta_u(p)=\langle p+u(p)\bn(p),\bn_u\rangle$, where $\bn_u$ is the normal to $M_u$.
		\item The speed function $w_u(p)=\langle e_{n+1},\bn_u\rangle^{-1}$ evaluated at the point $p+u(p)\bn(p)$.
	\end{itemize}
	
	The following \cite[Lemma A.3]{CM15_Lojasiewicz} is useful in the computation.
	
	\begin{Lm}(\cite[Lemma A.3]{CM15_Lojasiewicz})
		There are functions $w,\nu,\eta$ depending on $(p,s,y)\in M\times\R\times T_pM$ that are smooth for $|s|$ small and depend smoothly on $M$ so that
		\[
		w_u(p)=w(p,u(p),\nabla u(p))\text{, }\nu_u(p)=\nu(p,u(p),\nabla u(p))\text{ and }\eta_u(p)=\eta(p,u(p),\nabla u(p)).\]
		The ratio $\frac{w}{\nu}$ depends only on $p$ and $s$. Finally, the functions $w,\nu,\eta$ satisfy
		\begin{itemize}
			\item $w(p,s,0)\equiv 1$, $\partial_s w(p,s,0)=0$, $\partial_{y_\alpha}w(p,s,0)=0$, and $\partial_{y_\alpha}\partial_{y_\beta}w(p,0,0)=\delta_{\alpha\beta}$.
			\item $\nu(p,0,0)=1$; the only non-zero first and second order terms are $\partial_s\nu(p,0,0)=H(p)$, $\partial_{p_j}\partial_s\nu(p,0,0)=H_j(p)$, $\partial_s^2 \nu(p,0,0)=H^2(p)-|A|^2(p)$, and $\partial_{y_\alpha}\partial_{y_\beta}\nu(p,0,0)=\delta_{\alpha\beta}$.
			\item $\eta(p,0,0)=\langle p,\bn\rangle$, $\partial_s\eta(p,0,0)=1$, and $\partial_{y_\alpha}(p,0,0)=-p_\alpha$.
		\end{itemize}
	\end{Lm}
	A direct corollary computes the mean curvature $H_u$ (\cite[Corollary A.30]{CM15_Lojasiewicz}).
	\begin{Cor}
		The mean curvature $H_u$ of $M_u$ is given by
		\begin{equation*}
			H_u(p)=\frac{w}{\nu}[\partial_s\nu-\dv_M(\partial_{y_\alpha}\nu)],
		\end{equation*}
		where $\nu$ and its derivatives are all evaluated at $(p,u(p),\nabla u(p))$.
	\end{Cor}
	
	\subsection{Rescaled MCF as graphs}\label{SSRMCFGraph}
	We use the above settings to view an RMCF as graphs over another RMCF. Given an RMCF $\{M_t\}$ defined for $t\in[0,T]$, let $u(p,t):\cup_{t\in[0,T]}M_t\times \{t\}\to \R$ be a function. Similarly, we can define $M_{u,t}$ to be a family of hypersurfaces, each $M_{u,t}$ is a graph of $u(\cdot,t)$ over $M_t$. We will use $w(t,\cdot)$, $\nu(t,\cdot)$, $\eta(t,\cdot)$ to denote $w,\nu,\eta$ on each $M_t$ respectively.
	
	Most Lemmas in this section can be proved verbatim, so we omit the proof here.
	
	The following Lemma is a generalization of \cite[Lemma A.44]{CM15_Lojasiewicz}.
	\begin{Lm}\label{LmMu}
		The graphs $M_{u,t}$ flow by RMCF if and only if $u$ satisfies
		\begin{multline*}
			\partial_t u(p,t)=w(t,p,u(p,t),\nabla u(p,t))\left(\frac{1}{2}\eta(t,p,u(p,t),\nabla u(p,t))-\frac{1}{2}\eta(t,p,0,0)-H_u+H_0\right. \\
			\left.-u\langle(\nabla (H_0-\frac{1}{2}\langle p,\bn(p)\rangle),\bn_u\rangle\right):=\cM_t u.
		\end{multline*}
	\end{Lm}

	Compared to \cite[Lemma A.6]{CM15_Lojasiewicz}, the only extra term in this equation of graphs is
	\begin{equation*}
		\cP_t(u)=-u\langle(\nabla (H_0-\frac{1}{2}\langle p,\bn(p)\rangle),\bn_u\rangle+(-\frac{1}{2}\langle p,\bn\rangle+H_0).
	\end{equation*}
	Note $\cP_t=0$ if $M_t$ is a shrinker. We also obtain the linearization of $\cM_t u$ at $u=0$ similar to \cite[Corollary A.8]{ColdingMinicozzi19_Dynamics}:

	\subsection{Controlling the nonlinearity}
	
	Just like \cite[Section A.2]{ColdingMinicozzi19_Dynamics}, we also define the nonlinearity
	\begin{equation*}
		\cQ_t(u)=\cM_t u-L_t u.
	\end{equation*}
	Note our $\cQ_t$ is $\cQ$ defined by Colding-Minicozzi in \cite[Section A.2]{ColdingMinicozzi19_Dynamics} adding the term $\cP_t$ at every time $t$. In \cite[Proposition A.12]{ColdingMinicozzi19_Dynamics}, the authors express
	$$\mathcal Q(u)=\bar f(p,u,\nabla u)+\dv_\Sigma(\bar W(p,u,\nabla u))+\langle\nabla\bar h,\bar V\rangle$$
	with some properties $(P1$-$P4)$ therein, which are the main properties used in the proof of the $Q$-Lipschitz approximation property (Lemma 4.3 of \cite{ColdingMinicozzi19_Dynamics}). In our case, we follow \cite{ColdingMinicozzi19_Dynamics} to introduce $\bar f,\bar W$, $\bar h$ and $\bar V$ as functions of $t,p, u,\nabla u$. The only difference is that here $u$ is a function on $M_t$ while in \cite{ColdingMinicozzi19_Dynamics}, $u$ is a function on the fixed shrinker $\Sigma$. Next, we introduce $\tilde f=\bar f+\cP_t$, so that we have $$\mathcal Q_t(u)=\tilde f(t,p,u,\nabla u)+\dv_{M_t}(\bar W(t,p,u,\nabla u))+\langle\nabla\bar h,\bar V\rangle.$$
	
	We want to show that $\mathcal Q_t$ satisfies \cite[Proposition A.12]{ColdingMinicozzi19_Dynamics} with $\bar f$ replaced by $\tilde f$ here. It is enough to recover $(P1)$ in \cite[Proposition A.12]{ColdingMinicozzi19_Dynamics} since it is the only statement about $\bar f$. This is the content of the following lemma.
	\begin{Lm}
		$\tilde{f}$ satisfies
		$
		\tilde{f}(t,p,0,0)=\partial_s \tilde{f}(t,p,0,0)=\partial_{y_\alpha}\tilde{f}(t,p,0,0)=0.
		$	\end{Lm}

	With this, every calculation in \cite[Appendix A]{ColdingMinicozzi19_Dynamics} can be adapted to $\cM_t$ and $\cQ_t$. So we conclude the following corollary.
	
	\begin{Cor}
		\cite[Proposition A.12]{ColdingMinicozzi19_Dynamics}, \cite[Lemma A.24]{ColdingMinicozzi19_Dynamics}, hence \cite[Lemma 3.5]{ColdingMinicozzi19_Dynamics}, are valid for $\cM_t$ and $\cQ_t$.
	\end{Cor}
	
	Our final remark in this section is the following. All these functions discussed above are time-dependent, and they actually only depend on the geometry of the hypersurface $M_t$ at time $t$. The geometry of $M_t$ is uniformly controlled in our setting: when $t$ is large, $M_t$ is close to the limit shrinker $\Sigma$. Therefore ,we conclude that all the functions we discussed above are uniformly bounded.
	
	\subsection{Time derivative of integral under RMCF}
	We also want to process the analysis in \cite[Section 4]{ColdingMinicozzi19_Dynamics} in our setting. The argument is almost the same, except that there is a time derivative involved in the calculation. If the RMCF is a fixed shrinker, these kinds of computations have appeared in \cite{HuiskenPolden99_GeometricEvolution}, \cite{ColdingMinicozzi12_generic}, and \cite{Wang14_UniquenessConicalEnds}.
	
	\begin{Prop}\label{Prop:Appendix time derivative}
		Let $f(p,t):\cup_{t\in[0,\epsilon)}M_t\times\{ t\}\to \R$. Then
		\begin{equation*}
			\partial_t\int_{M_t}fe^{-\frac{|x|^2}{4}}d\mu_t=\int_{M_t}\left(\partial_t f-f(H-\frac{1}{2}\langle x,\bn\rangle)^2\right)e^{-\frac{|x|^2}{4}}d\mu_t.
		\end{equation*}
	\end{Prop}
	\begin{proof}
		By the RMCF equation, we have
		$	\partial_t x=-H\bn+\frac{1}{2}\langle x,\bn\rangle\bn,
		$
		and together with the first variational formula,
		$
		\partial_t(d\mu)=(-H+\frac{1}{2}\langle x,\bn\rangle)Hd\mu.
		$		So
		\begin{equation*}
			\begin{split}
				\partial_t\int_{M_t}fe^{-\frac{|x|^2}{4}}d\mu_t
				=&\int_{M_t}\partial_t fe^{-\frac{|x|^2}{4}}d\mu_t
				+
				\int_{M_t} -f\frac{1}{2}\langle -H\bn+\frac{1}{2}\langle x,\bn\rangle\bn,x\rangle e^{-\frac{|x|^2}{4}}d\mu_t
				\\
				&+
				\int_{M_t} fe^{-\frac{|x|^2}{4}}(-H+\frac{1}{2}\langle x,\bn\rangle)Hd\mu_t
				\\
				=&\int_{M_t}\left(\partial_t f-f(H-\frac{1}{2}\langle x,\bn\rangle)^2\right)e^{-\frac{|x|^2}{4}}d\mu_t.
			\end{split}
		\end{equation*}
	\end{proof}	
	For the integral of the inner product of the gradient of two vectors, the time derivative also introduces extra terms that come from the evolution of the metric tensor under the RMCF. So we need to compute the inner product under the evolution.
	
	Recall the shape operator $\cS:T_pM\to T_pM $ is defined to be the symmetric linear operator by
	$
	\langle \cS (X),Y\rangle=A(X,Y).
	$	\begin{Prop}\label{Prop:Appendix time derivative inner product}
		Let $f,h:M_t\to \R$. Then we have 
		\begin{equation*}
			\partial_t \nabla f=-(H-\frac{\langle x,\bn\rangle}{2})\cS(\nabla f)+\langle\nabla (-H+\frac{\langle x,\bn\rangle}{2}),\nabla f\rangle\bn,
		\end{equation*}
		\begin{equation*}
			\partial_t\langle\nabla f,\nabla h\rangle=-2(H-\frac{\langle x,\bn\rangle}{2})A(\nabla f,\nabla h).
		\end{equation*}
	\end{Prop}
	
	\begin{proof}
		Let us work in a local geodesic coordinate chart near the point $p\in M_t$ for fixed $t$. Then we have
		\begin{equation*}
			\partial_t\langle\nabla f,\nabla h\rangle
			=
			\partial_t(g^{ij}\partial_i f\partial_j h)
			=
			\partial_t(g^{ij})\partial_i f\partial_j h.
		\end{equation*}
		Suppose the RMCF is given by the map $F$ locally in short time. Then we have at $p$,
		\begin{equation*}
			\partial_t g_{ij}=2(H-\frac{\langle x,\bn\rangle}{2})a_{ij},\qquad
			\partial_t g^{ij}=-2(H-\frac{\langle x,\bn\rangle}{2})a_{ij}.
		\end{equation*}	
		As a result, we get 
		$			\partial_t\langle\nabla f,\nabla h\rangle=-2(H-\frac{\langle x,\bn\rangle}{2})A(\nabla f,\nabla h).$
		A similar computation gives the first identity.
	\end{proof}
	
	Together with Proposition \ref{Prop:Appendix time derivative}, we obtain the following time derivative.
	
	\begin{Prop}\label{Prop:Appendix time derivative inner product integral}
		\begin{equation*}
			\frac{1}{2}\partial_t\int_{M_t}|\nabla u|^2e^{-\frac{|x|^2}{4}}
			=
			\int_{M_t}\left(\langle\nabla u_t,u\rangle
			-
			(H-\frac{\langle x,\bn\rangle}{2})A(\nabla u,\nabla u)
			-
			\frac{1}{2}|\nabla u|^2(H-\frac{\langle x,\bn\rangle}{2})^2\right)e^{-\frac{|x|^2}{4}}
		\end{equation*}
	\end{Prop}
	
	Proposition \ref{Prop:Appendix time derivative} and Proposition \ref{Prop:Appendix time derivative inner product} imply the following corollary, which is an adaptation of Lemma 4.19 of \cite{ColdingMinicozzi19_Dynamics} to our setting.
	
	\begin{Cor}\label{Cor:Appendix section 4 of CM3 holds for rMCF}
		Given an RMCF $M_t$ which converges to a closed shrinker $\Sigma$ smoothly, then there exists $C$ so that if $u$ satisfies the equation
		$
		\partial_t u=L_{M_t} u
		$
		for $t\in[\hat{t},\hat{t}+1]$ for $\hat t$ sufficiently large, then we have
		\begin{equation*}
			\int_{M_t}|u(x,t)|^2 e^{-\frac{|x|^2}{4}}\leq e^{C(t-\hat t)}\int_{M_{\hat{t}}}|u(x,\hat{t})|^2e^{-\frac{|x|^2}{4}},
		\end{equation*}
		\begin{equation*}
			\int_{0}^1\int_{M_{t}}|\nabla u(x,t)|^2 e^{-\frac{|x|^2}{4}} \leq C\int_{M_{\hat{t}}}|u(x,\hat{t})|^2e^{-\frac{|x|^2}{4}}.
		\end{equation*}
		
	\end{Cor}
	
	\begin{proof}
		We first have the calculation 
		\begin{equation*}
			\begin{aligned}
				\frac12\partial_t\int_{M_t}|u(x,t)|^2 e^{-\frac{|x|^2}{4}}&=\int_{M_{t}}u\partial_t u-(H-\frac12\langle x,\mathbf n\rangle)^2u^2e^{-\frac{|x|^2}{4}}\\
				&\leq \int_{M_{t}}u L_{M_t}ue^{-\frac{|x|^2}{4}}\\
				&=\int_{M_{t}}-|\nabla u|^2+(|A|^2+\frac12) u^2 e^{-\frac{|x|^2}{4}}.
			\end{aligned}
		\end{equation*}
		Then both item follows since we have a uniform bound of $|A|^2$. 
	\end{proof}
	Finally, we compute the evolution of the Laplacian along the RMCF. It will be used in the next section.
	\begin{Prop}\label{Prop: Appendix time derivative of Laplacian}
		\begin{equation*}
			\partial_t(\Delta)=
			2(H-\frac{\langle x,\bn\rangle}{2})\dv\cS(\nabla \cdot)
			+
			2\langle\nabla (H-\frac{\langle x,\bn\rangle}{2}),\nabla \cdot\rangle
			-
			\langle\nabla [(H-\frac{\langle x,\bn\rangle}{2})H],\nabla \cdot\rangle.
		\end{equation*}
	\end{Prop}
	
	\begin{proof}
		Let $f,h$ be two smooth functions on $M_t$. Then
		\begin{equation*}
			\int_{M_t}f\Delta h d\mu=-\int_{M_t}\langle\nabla f,\nabla h\rangle d\mu.
		\end{equation*}
		Take time derivative on both sides
		\begin{equation*}
			\begin{split}
				&\int_{M_t}f(\Delta)_t h d\mu-\int_{M_t}f\Delta h(H-\frac{\langle x,\bn\rangle}{2})H d\mu
				\\
				=&
				-2\int_{M_t}(H-\frac{\langle x,\bn\rangle}{2})A(\nabla f,\nabla h)d\mu
				+
				\int_{M_t}\langle\nabla f,\nabla h\rangle(H-\frac{\langle x,\bn\rangle}{2})H d\mu
				\\
				=&
				-2\int_{M_t}(H-\frac{\langle x,\bn\rangle}{2})\langle \nabla f,\cS(\nabla h)\rangle d\mu
				+
				\int_{M_t}\langle\nabla f,\nabla h\rangle(H-\frac{\langle x,\bn\rangle}{2})H d\mu .
			\end{split}
		\end{equation*}
		
		Integration by parts gives
		\begin{multline*}
			\int_{M_t}f(\Delta)_t h
			=
			2\int_{M_t} f \langle\nabla h,\nabla (H-\frac{\langle x,\bn\rangle}{2})\rangle
			+
			2\int_{M_t} f (H-\frac{\langle x,\bn\rangle}{2})\dv\cS(\nabla h)
			\\
			-
			\int_{M_t}f\langle\nabla h,\nabla [(H-\frac{\langle x,\bn\rangle}{2})H]\rangle.
		\end{multline*}
		Since the test functions $f,h$ are arbitrary, we conclude that
		\begin{equation*}
			\partial_t(\Delta)=
			2(H-\frac{\langle x,\bn\rangle}{2})\dv\cS(\nabla \cdot)
			+
			2\langle\nabla (H-\frac{\langle x,\bn\rangle}{2}),\nabla \cdot\rangle
			-
			\langle\nabla [(H-\frac{\langle x,\bn\rangle}{2})H],\nabla \cdot\rangle.
		\end{equation*}
	\end{proof}
	\begin{Rk}
		In local coordinates, one can check that
		$
		\dv\cS(\nabla f)=a_{ij}f_{ij}.
		$
		Therefore if $(H-\frac{\langle x,\bn\rangle}{2})$ is sufficiently small (i.e. $M_t$ is sufficiently closed to $\Sigma$), we have
		\begin{equation*}
			|2(H-\frac{\langle x,\bn\rangle}{2})\dv\cS(\nabla f)|\leq \frac{1}{2}|\Hess f|.
		\end{equation*}
	\end{Rk}

	\section{Li-Yau estimate and Harnack inequality for linearized RMCF}\label{SLiYau}
	In this section, we follow the famous Li-Yau estimate in \cite{LiYau86_parabolic_kernel} to develop a Harnack inequality for positive solutions to the linearized equation of RMCF.
	
	\subsection{Generalized Li-Yau estimate}
	In this section, we recall a generalized Li-Yau estimate developed by Paul Lee in \cite{Lee17_GeneLiYau}. This estimate admits the first-order term and zeroth-order term in the heat equation on a Ricci non-negative manifold. In our application, the shrinker may have negative (but bounded from below) Ricci curvature. So we need to slightly improve the Theorem. We do not need the sharpness of \cite[Theorem 1.1]{Lee17_GeneLiYau}, so our statement does not take care of the precise values of the constants.
	
	\begin{Thm}[Improvement of \cite{Lee17_GeneLiYau} Theorem 1.1]\label{Thm: Generalized Li-Yau}
		Assume the Ricci curvature of a closed Riemannian manifold $M$ is bounded from below by $-K$. Let $U_1,U_2$ be two smooth functions on $M$, and let
		$V=\Delta U_1+\frac{1}{2}|U_1|^2-2U_2.$
		Assume $|V|$, $|\nabla V|$ and $|\Delta V|$ are both bounded by $k$. Suppose $u$ is a positive solution of the equation
		$
		u_t=\Delta u +\langle\nabla U_1,\nabla u\rangle+U_2 u.
		$
		Then there is $\tau>0$ depending on $k,K$ and the geometry of $M$, such that $u$ satisfies
		\begin{equation*}\label{Eq: Generalized Harnack}
			\partial_t (\log u)\geq \frac{|\nabla (\log u)|^2}{2}-\frac{|\nabla U_1|^2}{8}-\frac{V}{4}-\frac{n}{t}-2nK-k/2\quad\text{on $(0,4\tau)$.}
		\end{equation*}
		
	\end{Thm}
	
	\begin{proof}
		We follow the idea in \cite[Section 9]{Lee17_GeneLiYau}. Define
		$
		f=\log u+\frac{1}{2}U_1.
		$
		Then
		\begin{equation*}
			\begin{split}
				f_t=
				&
				|\nabla f|^2+U_2-\frac{1}{4}|\nabla U_1|^2-\frac{1}{2}\Delta U_1+\Delta f
				=
				|\nabla f|^2+\Delta f-\frac{1}{2}V,\\
				f_{tt}=&\Delta f_t+2\langle \nabla f,\nabla f_t\rangle.
			\end{split}
		\end{equation*}
		
		We also compute that
		\begin{equation}\label{Eq: Appendix Harnack - time derivative of |nabla f|^2}
			\begin{split}
				&\partial_t|\nabla f|^2-\Delta|\nabla f|^2-2\langle\nabla f,\nabla |\nabla f|^2\rangle\\
				=&
				2\langle \nabla f_t,\nabla f\rangle-\Delta|\nabla f|^2-2\langle\nabla f,\nabla |\nabla f|^2\rangle\\
				=&
				-2|\Hess f|^2-2\Ric(\nabla f,\nabla f)-\langle\nabla V,\nabla f\rangle\\
				\leq &
				-\langle\nabla V,\nabla f\rangle
				-\frac{2}{n}(\Delta f)^2
				+2K|\nabla f|^2\\
				=&
				-\langle\nabla V,\nabla f\rangle
				-\frac{2}{n}\left( f_t-|\nabla f|^2+\frac{1}{2}V \right)^2
				+2K|\nabla f|^2
			\end{split}
		\end{equation}
		Define
		$
		F=-2 a f_t+ a|\nabla f|^2 -\frac{a}{2}V +b,
		$
		where $a,b$ to be determined, but we assume $a>0$ and $b\leq 0$ and we will finally choose $a,b$ satisfying these conditions. 
		
		Then we compute that
		\begin{equation*}
			\begin{split}
				F_t-\Delta F-2\langle \nabla f,\nabla F\rangle
				\leq&
				-a\langle\nabla V,\nabla f\rangle -\frac{2a}{n}\left( f_t-|\nabla f|^2+\frac{1}{2}V \right)^2
				+2aK|\nabla f|^2\\
				&+\frac{a}{2}\Delta V
				+a\langle \nabla f,\nabla V\rangle
				-2 a'f_t
				+a'|\nabla f|^2
				-\frac{a'}{2}V
				+b'\\
				=&
				-\frac{2a}{n}\left( \frac{1}{2 a}F+\frac{1}{2}|\nabla f|^2-\frac{b}{2 a}-\frac{V}{4}\right)^2
				+2aK|\nabla f|^2\\
				&+\frac{a}{2}\Delta V
				+b'
				+\frac{a'}{a}F-\frac{a'b}{a}
				\\
				=&
				-\frac{2a}{n}\left( \frac{1}{2 a}F+\frac{1}{2}|\nabla f|^2-\frac{b}{2 a}-\frac{V}{4}-nK\right)^2\\
				&-
				2K F
				+
				2bK
				+
				2anK^2
				+\frac{a}{2}\Delta V
				+b'
				+\frac{a'}{a}F-\frac{a'b}{a}+aKV\\
			\end{split}
		\end{equation*}
		Now we pick
		$
		a=e^{(2K-1)t},\ b=-\frac{2ne^{(2K-1)t}}{t}.
		$
		Note, then if we let
		\[g=\frac{b}{a}=-\frac{2n}{t},\quad g'=\frac{2n}{t^2}=\frac{1}{2n}g^2\]
		Then we have two cases:
		
		{\bf Case 1.} If $\max F\leq 2a(nK+k/4)=2(nK+k/4) e^{(2K-1)t}$ on $(0,4\tau]$ for some $\tau$ to be determined, then $F$ satisfies the desired bound.
		
		{\bf Case 2.} If $\max F>2a(nK+k/4)$ on $(0,4\tau]$, then
		\begin{equation*}
			\begin{split}
				F_t-\Delta F-2\langle\nabla f,\nabla F\rangle
				\leq &
				-\frac{2a}{n}\left(\frac{b}{2a}\right)^2
				+
				\left(\frac{a'}{a}-2K\right)F
				+
				aKk
				+
				2anK^2
				+
				\frac{a}{2}k
				+2bK
				+
				b'
				-
				\frac{a'b}{a}.
			\end{split}
		\end{equation*}
		Consider $F$ achieves its maximum at some point. Since $\max F>2a(nK+k/4)>0$, and $\lim_{t\to 0}b=-\infty$, this maximum point must be achieved at somewhere such that
		\[\nabla F=0,\quad \Delta F\leq 0,\quad \partial_t F\geq 0.\]
		Therefore, at this maximum point,
		\begin{equation*}
			-\frac{2a}{n}\left(\frac{b}{2a}\right)^2
			+
			\left(\frac{a'}{a}-2K\right)F
			+
			aKk
			+
			2anK^2
			+
			\frac{a}{2}k
			+2bK
			+
			b'
			-
			\frac{a'b}{a}\geq 0.
		\end{equation*}
		Recall that $g=\frac{b}{a}$. Plug in $a,b$ and divide the equation by $a$ we get
		\begin{equation*}
			-\frac{F}{a}+(Kk+2nK^2+\frac{k}{2})-\frac{1}{2n}g^2+2Kg+g'\geq 0,
		\end{equation*}
		Hence
		\begin{equation*}
			\frac{F}{a}\leq (Kk+2nK^2+\frac{k}{2})+2Kg.
		\end{equation*}
		Again, since $\lim_{t\to 0}g(t)=-\infty$, so there is a time interval $(0,4\tau]$ such that on this interval, $(Kk+2nK^2+\frac{k}{2})+2Kg\leq 0$. This is the desired $\tau$. Then if we restricted to this interval, at the maximum point, $F\leq 0$, which is a contradiction to $F>2anK$. So we conclude the proof.
		
	\end{proof}
	
	\begin{Rk}
		Although this is not important in our application, we can make
		$
		4\tau=\frac{2Kk+4nK^2+k}{8Kn}
		$
		in the statement of the theorem.
	\end{Rk}
	
	\begin{Cor}\label{Cor: Harnack bound different time}
		Given the assumptions in Theorem \ref{Thm: Generalized Li-Yau}, we also assume $|V|$, $|\nabla U_1|$ are uniformly bounded by a constant $C_1$. Then there is a uniform constant $C$ only depending on the geometry of $M$, $K$, $k$, $C_1$ and $\tau$, such that for $x_1,x_2\in M$, $\tau\leq t_1\leq t_2\leq 4\tau$, $t_2-t_1\geq \tau$,
		\begin{equation*}
			u(x_1,t_1)\leq C u(x_2,t_2).
		\end{equation*}
	\end{Cor}

	\begin{proof}
		We pick a shortest geodesic $\gamma$ connecting $x_1,x_2$, parametrized by $[t_1,t_2]$. Since the diameter of the manifold is uniformly bounded by some constant, $|\dot{\gamma}|$ is uniformly bounded by some constant depending on the diameter of $M$ and $\tau$. Then we compute
		\begin{equation*}
			\begin{split}
				\frac{d}{dt}\log(u(\gamma(t),t))
				=&\frac{\partial}{\partial t}(\log u(\gamma(t),t))
				+
				\langle\nabla \log u(\gamma(t),t),\dot\gamma\rangle\\
				\text{by (\ref{Eq: Generalized Harnack})}
				\geq &
				\frac{|\nabla (\log u)|^2}{2}-\frac{|\nabla U_1|^2}{8}-\frac{V}{4}-\frac{n}{t}+2nK
				+\langle\nabla \log u(\gamma(t),t),\dot\gamma\rangle
				\\
				\text{Cauchy-Schwartz}
				\geq &
				\frac{|\nabla (\log u)|^2}{2}-C_2-\frac{n}{t}
				-\frac{|\nabla (\log u)|^2}{2}-|\dot{\gamma}|^2
				\\
				\geq &
				-\frac{n}{t}
				-C_2
			\end{split}
		\end{equation*}
		Here, $C_2$ is a uniformly bounded constant. Then we conclude that
		\begin{equation*}
			\log(u(x_2,t_2))-\log(u(x_1,t_1))
			=
			\int_{t_1}^{t_2}\frac{d}{ds}\log (u(s,\gamma(s)))
			\geq
			-n(\log(t_2)-\log t_1)
			-C_2(t_2-t_1).
		\end{equation*}
		Since $t_1,t_2\in[\tau,4\tau]$, this number is bounded from below by some constant $C_3$. Thus,
		\begin{equation*}
			\frac{u(x_2,t_2)}{u(x_1,t_1)}\geq e^{C_3}=C.
		\end{equation*}
		Where $C$ is a constant depending on the geometry of $M$, $K$, $k$, $C_1$.
	\end{proof}
	
	Finally, we will apply this Harnack inequality to the linearized RMCF on a shrinker.
	\begin{Cor}\label{Cor: Harnack Bound for fixed time slice}
		Suppose $\Sigma$ is a closed embedded shrinker. Then there exists a constant $\eps$ such that the following is true: Suppose $M$ is a graph of function $v$ over $\Sigma$, with $\|v\|_{C^4}\leq \eps$. Then there exist $\tau$ and $C$, such that for any positive solutions $u$ on $M\times[0,4\tau]$ to the  equation
		$	\partial_t u=Lu,
		$
		and any $x_1,x_2\in M$, $\tau\leq t_1\leq t_2\leq 4\tau$, $t_2-t_1\geq \tau$, we have
		\begin{equation*}
			u(x_1,t_1)\leq C u(x_2,t_2).
		\end{equation*}
	\end{Cor}
	
	\begin{proof}
		Pick $U_1=-\frac{|x|^2}{4}$ and $U_2=\frac{1}{2}+|A|^2$, then we can apply Theorem \ref{Thm: Generalized Li-Yau} and Corollary \ref{Cor: Harnack bound different time} to the equation $\partial_t u=Lu$. Since $v$ has uniformly bounded $C^4$-norm, up to second derivative of $U_1$ and $U_2$ are bounded (hence $V$ is bounded), as well as the Ricci curvature of $M$, the diameter of $M$, are all uniformly bounded. Then the statement follows from Corollary \ref{Cor: Harnack bound different time}.
	\end{proof}

	\subsection{Time dependence analysis to RMCF}
	We want to get a Harnack inequality for our linearized RMCF on a given RMCF. Hence, we want to generalize the above Harnack inequality to a time-dependent heat equation.
	
	From now on, we will assume our manifold has some uniformly bounded geometry, i.e., its diameter, curvature, ... are uniformly bounded by some constant. Then the equation we want to study is
	\begin{equation*}
		u_t=\Delta_t u+\langle \nabla U_1,\nabla u\rangle+U_2 u,
	\end{equation*}
	where $U_1=-\frac{|x|^2}{4}$, $U_2=(\frac{1}{2}+|A|^2)$, which are both time-dependent. The Laplacian, the inner product, and the gradient are also time-dependent, and we have already computed them in Proposition \ref{Prop:Appendix time derivative}, Proposition \ref{Prop:Appendix time derivative inner product}, and Proposition \ref{Prop: Appendix time derivative of Laplacian}.
	
	We will always assume our analysis is on the RMCF $M_t$ for $t$ sufficiently large. Therefore, we will assume all the derivatives up to the second order of $U_1,U_2$, curvature, etc., are uniformly bounded. Recall the Simon type inequality of curvature flow (See \cite[Lemma 7.6]{HuiskenPolden99_GeometricEvolution}) shows that the time derivative of the curvature terms can be bounded by the higher-order (space) derivative of the curvature terms. Therefore, we may also assume the time derivatives of these quantities are uniformly bounded when $t$ is sufficiently large.
	
	\begin{Thm}
		Let $u$ be a positive solution of the linearized equation $\partial_t u=L_{M_t} u$ where $(M_t)$ is an RMCF converging to a compact shrinker $\Sigma$. Then there exists $\widetilde t>0$ sufficiently large, $\tau>0$ and $C>0$ such that for all $T>\widetilde t$ and all $t\in (T,T+4\tau)$ we have 
		\begin{equation*}
			\partial_t (\log u)\geq \frac{|\nabla \log u|^2}{2}-C-\frac{C}{t-T} .
		\end{equation*}
	\end{Thm}
	
	\begin{proof}
		The proof is the same; here we just point out the necessary modifications. Again we define
		$
		f=\log u+\frac{1}{2}U_1.
		$
		and $V=\Delta U_1 +\frac{1}{2}|U_1|^2-2U_2-\partial_t U_1$. Then we still have
		\begin{equation*}
			f_t=|\nabla f|^2+\Delta f-\frac{1}{2}V.
		\end{equation*}
		However, we do not have $f_{tt}=\Delta f_t+2\langle\nabla f,\nabla f_t\rangle$. Instead, we will have an extra term coming from the time derivative:
		\begin{equation*}
			\begin{split}
				f_{tt}-\Delta f_t-2\langle \nabla f,\nabla f_t\rangle
				= &
				(\Delta)' f-\frac{1}{2}V_t-2A(\nabla f,\nabla f)
				\\
				\leq &
				C_1|\nabla f|^2+\frac{1}{2}|\Hess f|^2+C_2\\
			\end{split}
		\end{equation*}
		
		Here we use Cauchy-Schwartz, Proposition \ref{Prop: Appendix time derivative of Laplacian}, and notice the Remark after it. Similarly we have (c.f. (\ref{Eq: Appendix Harnack - time derivative of |nabla f|^2}))
		\begin{equation*}
			\begin{split}
				\partial_t|\nabla f|^2-\Delta|\nabla f|^2-2\langle\nabla f,\nabla|\nabla f|^2\rangle
				\leq &
				-\langle\nabla V,\nabla f\rangle
				-2|\Hess f|^2
				+2K|\nabla f|^2+C_3|\nabla f|^2
				\\
				\leq &
				-\langle\nabla V,\nabla f\rangle
				-2|\Hess f|^2
				+C_4|\nabla f|^2.
			\end{split}
		\end{equation*}
		Here we use Proposition \ref{Prop:Appendix time derivative inner product}. Therefore, if we define
		\begin{equation*}
			F=-2 a f_t+ a|\nabla f|^2 -\frac{a}{2}V +b
		\end{equation*}
		again, then we have a similar inequality
		\begin{multline*}
			F_t-\Delta F-2\langle \nabla f,\nabla F\rangle
			\\
			\leq
			-a\langle\nabla V,\nabla f\rangle
			-\frac{a}{n}\left( f_t-|\nabla f|^2+\frac{1}{2}V \right)^2
			+C_5 a|\nabla f|^2
			+(C_6+2C_2) a
			+b'
			+\frac{a'}{a}F-\frac{a'b}{a}.
		\end{multline*}
		Then by the same argument as in the proof of Theorem \ref{Thm: Generalized Li-Yau}, we can prove that
		\begin{equation*}
			F\leq C_7 a, \quad a=e^{C_8 t}, \quad b=-C_9\frac{1}{t-T}.
		\end{equation*}
		for some constants $C_7,C_8,C_9$. So we obtain a Harnack inequality similar to (\ref{Eq: Generalized Harnack}):
		\begin{equation*}
			\partial_t(\log u)\geq |\nabla(\log u)|^2-C_{10}-C_{9}\frac{1}{t-T}.
		\end{equation*}
	\end{proof}
	
	Then exactly the same argument as in Corollary \ref{Cor: Harnack bound different time} shows the following Corollary for linearized RMCF.
	
	\begin{Cor}\label{Cor: Appendix Harnack to rMCF}
		Let $M_t$ be an RMCF converging to $\Sigma$ smoothly. There exist $\widetilde{t}>0$, $\tau>0$, $C>0$ such that the following is true: Suppose $u$ is a positive solution to the RMCF equation
		$\partial_t=L_t u$
		on the time interval $[t,t+4\tau]$ for $t\geq \widetilde{t}$, then for $x_1\in M_{t_1},\ x_2\in M_{t_2}$ and $t+\tau\leq t_1\leq t_2\leq t+4\tau$, $t_2-t_1\geq \tau$, we have
		\begin{equation*}
			u(x_1,t_1)\leq C u(x_2,t_2).
		\end{equation*}
	\end{Cor}
	
	Together with Lemma 4.19 of \cite{ColdingMinicozzi19_Dynamics} in the RMCF version (see Proposition \ref{Cor:Appendix section 4 of CM3 holds for rMCF}), we can prove the following important averaging property of the positive solutions to the linearized RMCF.
	
	\begin{Thm} \label{Thm:Appendix Averaging property of rMCF}
		Let $M_t$ be an RMCF converging to $\Sigma$ smoothly. there exist $\widetilde{t}>0$, $C>0$ such that the following is true: Suppose $u$ is a positive solution to the linearized RMCF equation
		$\partial_t u=L_{M_t} u$, then we have for $t>\widetilde t$
		\begin{equation*}
			\max u(\cdot,t)\leq C\min u(\cdot,t).
		\end{equation*}
	\end{Thm}
	
	\begin{proof}[Proof of Theorem \ref{Thm:Appendix Averaging property of rMCF}]
		In the proof, $C$ may vary from line to line, but it is always some universal constant.
		First, we assume $t$ is sufficiently large so that we can use Corollary \ref{Cor: Appendix Harnack to rMCF}. Then for $s\in[t+2\tau,t+3\tau]$, assume
		\[\max u(\cdot,s)=u(y,s),\quad \min u(\cdot,s)=u(z,s).\]
		Then by integrating $u(x,t+\tau)\leq C u(z,s)$ we obtain
		\begin{equation*}
			\int_{M_{t+\tau}}|u(x,t+\tau)|^2e^{-\frac{|x|^2}{4}}\leq C (u(z,s))^2.
		\end{equation*}
		Monotonicity formula (c.f. Proposition \ref{Cor:Appendix section 4 of CM3 holds for rMCF}) implies that
		\begin{equation*}
			\int_{M_{t+4\tau}}|u(x,t+4\tau)|^2e^{-\frac{|x|^2}{4}}\leq C
			\int_{M_{t+\tau}}|u(x,t+\tau)|^2e^{-\frac{|x|^2}{4}}.
		\end{equation*}
		By integrating $u(z,s)\leq C u(x,t+4\tau)$ we obtain
		\begin{equation*}
			(u(y,s))^2\leq C\int_{M_{t+4\tau}}|u(x,t+4\tau)|^2e^{-\frac{|x|^2}{4}}.
		\end{equation*}
		So we conclude 
		$			u(y,s)\leq C u(z,s).
		$	\end{proof}
	
	\section{Closeness of graphs}\label{S:Closeness of graphs}
	In this appendix, we prove Theorem \ref{Thm:closeness of graphs} and discuss the properties of the transplanted functions. Since the Banach space $\cE_{1+\al}$ is equivalent to the Banach space $C^{2,2\al}(\Sigma)$ (see Proposition \ref{PropE}, we only need to prove the Theorem \ref{Thm:closeness of graphs} for $C^{2,\alpha}$ norm.
	
	Let us recall the statement. Let $\Sigma$ be a closed embedded hypersurface, let $\Sigma^f$ be the hypersurface $\{x+f(x)\bn(x):x\in\Sigma\}$. Suppose $g\in C^{2,\alpha}(\Sigma^f)$, then we use $\bar{g}$ to denote the function transplant to $C^{2,\alpha}(\Sigma)$, i.e.
	$\bar{g}(x)=g(x+f(x)\bn(x)).$
	
	The following Lemma implies that some function norms would not change much after the transplantation.
	\begin{Lm}\label{Lm:Appendix Transplant functions on graphs}
		Given $\eps>0$ sufficiently small, there exists $C>0$ depending only on $\Sigma$ such that the following is true. If $\|f\|_{C^4}\leq \eps$, then
		\begin{equation*}
			(1-\eps C)\|g\|_{C^{2,\alpha}(\Sigma^f)}\leq \|\bar{g}\|_{C^{2,\alpha}(\Sigma)}\leq (1+\eps C)\|g\|_{C^{2,\alpha}(\Sigma^f)}.
		\end{equation*}
	\end{Lm}

	\begin{proof}
		We can identify $\Sigma^f$ with $\Sigma$ by sending $x+f(x)\bn(x)$ to $x$. Then $\Sigma^f$ and $\Sigma$ are two metrics on the same manifold. Standard graph estimate (c.f. Appendix A) shows that the metric, the gradient, and the second-order derivative operator are all $C^0$ closed on $\Sigma^f$ and $\Sigma$. Therefore, we obtain the desired closeness in the statement of the Lemma.
	\end{proof}
	
	\begin{Thm}\label{Thm:Appendix closeness of graphs}
		Let $\Sigma$ be a fixed embedded closed hypersurface. Then given $\eps>0$, there exists a constant $\mu(\eps)>0$ such that the following is true: Suppose $\Sigma_1$ is the graph $\{x+f\bn(x):x\in\Sigma\}$ over $\Sigma$, and $\Sigma_2$ is the graph $\{y+g\bn(y):y\in\Sigma_1\}$ over $\Sigma_1$, and $\|f\|_{C^4(\Sigma)}\leq \mu $, $\|g\|_{C^{2,\alpha}(\Sigma_1)}\leq \mu$, then $\Sigma_2$ is a graph of a function $v$ on $\Sigma$,
		and
		\begin{equation*}
			\|v-(f+g)\|_{C^{2,\alpha}(\Sigma)}\leq \eps\|g\|_{C^{2,\alpha}(\Sigma_1)}.
		\end{equation*}
		Here we transplant $g$ on $\Sigma_1$ to a function on $\Sigma$, and still use $g$ to denote it.
	\end{Thm}

	\begin{proof}
		In the proof, the constant $C$ may vary line to line, but only depending on $\Sigma$ and $\mu$. We will first fix a tubular neighborhood $\cN$ of $\Sigma$ such that the projection map is well-defined. Let us use $\Pi:\cN\to\Sigma$ to denote this projection, and assume $\|\Pi\|_{C^3(\cN)}$ is bounded by a constant $C$.

		We will use $\bn^f$ to denote the unit normal vector on $\Sigma_1=\Sigma+f\bn$. From now on, we will transplant every function ($g$, normal vectors, etc.) on $\Sigma_1$ to $\Sigma$ by Lemma \ref{Lm:Appendix Transplant functions on graphs}, and we will drop the bar notation if there is no ambiguity. We will assume $\|f\|_{C^2},\|g\|_{C^2}$ are sufficiently small such that $\Sigma +f\bn$, $\Sigma+ (f+g)\bn$ and $(\Sigma+f\bn)+g\bn^{f}$ belongs to $\cN$. Then we have
		\begin{equation*}
			\begin{split}
				&
				v\left(\Pi((x+f(x)\bn(x))+g(x)\bn^{f}(x))\right)
				\\
				&=
				\left\langle x+f(x)\bn(x)+g(x)\bn^{f}(x), \bn\left(\Pi(x+f(x)\bn(x)+g(x)\bn^{f}(x))\right)\right\rangle
				\\
				&-
				\left\langle \Pi(x+f(x)\bn(x)+g(x)\bn^{f}(x)), \bn\left(\Pi(x+f(x)\bn(x)+g(x)\bn^{f}(x))\right)\right\rangle.
			\end{split}
		\end{equation*}
		Note that
		\begin{equation*}
			\begin{split}
				(f+g)(x)
				=&
				\left\langle x+f(x)\bn(x)+g(x)\bn(x), \bn\left(\Pi(x+f(x)\bn(x)+g(x)\bn(x))\right)\right\rangle
				\\
				&-
				\left\langle \Pi(x+f(x)\bn(x)+g(x)\bn(x)), \bn\left(\Pi(x+f(x)\bn(x)+g(x)\bn(x))\right)\right\rangle.
			\end{split}
		\end{equation*}
		If we write $x'=\Pi(x+f(x)\bn(x)+g(x)\bn^{f}(x))$, and we write \begin{equation*}
			\begin{split}
				F(x,V)
				=&
				\left\langle x+f(x)\bn(x)+V, \bn\left(\Pi(x+f(x)\bn(x)+V)\right)\right\rangle
				\\
				&-
				\left\langle \Pi(x+f(x)\bn(x)+V), \bn\left(\Pi(x+f(x)\bn(x)+V)\right)\right\rangle.
			\end{split}
		\end{equation*}
		Then by the fundamental theorem of calculus, we have
		\begin{equation*}
			v(x')-(f+g)(x)=\int_0^1\langle \nabla_V F(x,(s\bn^f(x)+(1-s)\bn(x))g(x)) , g(x)(\bn^f(x) -\bn(x))\rangle ds.
		\end{equation*}
		Thus,
		\begin{equation*}
			|v\left(x'\right)-(f+g)(x)|\leq C|g(x)|\|\bn^f-\bn\|_{C^0}\leq C\mu\|g\|_{C^0}.
		\end{equation*}
		Here we use the closeness of $\|\bn^f-\bn\|_{C^1}$ controlled by $\|f\|_{C^2}$ from Appendix A (the quantity $w$ in Lemma A.1).
		Next we write
		$
		G(x,V)=(f+g)(\Pi(x+f(x)\bn(x)+V)),
		$
		Then the fundamental theorem of calculus implies that
		\begin{equation*}
			(f+g)(x)-(f+g)(x')=\int_0^1\langle \nabla_V G(x,(s\bn^f(x)+(1-s)\bn(x))g(x)), g(x)(\bn^f(x) -\bn(x))\rangle ds.
		\end{equation*}
		Thus we have
		\begin{equation*}
			\begin{split}
				|(f+g)(x)-(f+g)\left(x'\right)|\leq
				C\mu \|\bn^f-\bn\|_{C^1}|g(x)|
				\leq
				C\mu\|g\|_{C^0}.
			\end{split}
		\end{equation*}
		Therefore, we conclude that
		\begin{equation*}
			|v\left(\Pi((x+f(x)\bn(x))+g(x)\bn^{f}(x))\right)-(f+g)\left(\Pi((x+f(x)\bn(x))+g(x)\bn^{f}(x))\right)|\leq C\mu\|g\|_{C^0}.
		\end{equation*}
		
		To get a higher order bound, we first notice that when $x$ is close to $y$,
		\begin{equation*}
			|(x-y)-(\Pi((x+f(x)\bn(x))+g(x)\bn^{f}(x))-\Pi((y+f(y)\bn(y))+g(y)\bn^{f}(y)))|\leq C\mu|x-y|.
		\end{equation*}
		Therefore, when $\mu$ is sufficiently small, we have
		$
		|x-y|\approx |x'-y'|.
		$
		Thus, if we use the distance defined on the ambient Euclidean space, the distance does not change too much. Thus, we can differentiate the above expressions of the fundamental theorem of calculus in the Euclidean space. Then standard composition property of H\"older norm implies the $C^{2,\alpha}$-estimate
		\begin{equation*}
			\begin{split}
				\|v\left(\Pi((x+f(x)\bn(x))+g(x)\bn^{f}(x))\right)-&(f+g)\left(\Pi((x+f(x)\bn(x))+g(x)\bn^{f}(x))\right)\|_{C^{2,\alpha}}
				\\
				&\leq C\mu\|g\|_{C^{2,\alpha}} 
			\end{split}
		\end{equation*}
		and a further smaller $\mu$ with Lemma \ref{Lm:Appendix Transplant functions on graphs} shows the desired estimate.
	\end{proof}

	\bibliographystyle{alpha}
	\bibliography{GMT}
\end{document}